\documentclass{amsart}
\usepackage{fouriernc} 

\usepackage{Definitions}    
\usepackage{PageSetup}      
\usepackage{Environments-Cary}   
\usepackage{array}
\usepackage{graphicx}
\usepackage{comment}
\usepackage{subcaption}
\usepackage{mathtools}

\usepackage[shortlabels]{enumitem}
\usepackage{xr}

\usepackage{tikz-cd} 
\usepackage{multirow}

\newcommand{\Osp}{\mathcal{OS}}
\newcommand{\cat}[1]{\textup{\textbf{{#1}}}}

\newcommand{\pbc}{\nu}
\newcommand{\shift}{\Gamma}
\newcommand{\fibrant}{radiant }

\DeclareMathOperator{\THH}{\textup{THH}}
\DeclareMathOperator{\TR}{\textup{TR}}
\providecommand{\Ex}{\cE\mathrm{x}}
\DeclareMathOperator{\barsmash}{\overline{\wedge}}
\newcommand{\Sph}{\mathbb{S}}
\newcommand{\sh}[1]{{\ensuremath{\hspace{1mm}\makebox[-1mm]{$\langle$}\makebox[0mm]{$\langle$}\hspace{1mm}{#1}\makebox[1mm]{$\rangle$}\makebox[0mm]{$\rangle$}}}}
\newcommand{\xto}{\xrightarrow}
\newcommand{\odots}[1]{\odot}
\newcommand{\bcr}[3]{\left[\begin{gathered}{#3}\xleftarrow{#2}{#1}\end{gathered}\right]}
\newcommand{\simar}{\overset\sim\to}
\newcommand{\R}{\mathbb{R}}
\newcommand{\ti}{\widetilde}
\newcommand{\ra}{\longrightarrow}
\newcommand{\mc}{\mathcal}

\makeatletter
\@namedef{subjclassname@2020}{\textup{2020} Mathematics Subject Classification}
\makeatother



\title[Smash powers and traces on parametrized spectra]{Smash powers and traces on parametrized spectra:\\An application of rigidity}
\author{Cary Malkiewich} 
\email{malkiewich@math.binghamton.edu}
\address{Binghamton University, PO Box 6000, Binghamton, NY 13902}
\author{Kate Ponto}
\email{kate.ponto@uky.edu}
\address{University of Kentucky, 719 Patterson Office Tower, Lexington, KY 40506}
\date{\today}

\begin{document}

 \maketitle

\begin{abstract}
	The bicategory of parametrized spectra has a remarkably rich structure. We can take traces in this bicategory, giving classical invariants that count fixed points. We can also take $C_n$-equivariant external smash powers and equivariant traces, which give significant generalizations of the classical invariants that count periodic points.
	
	Unfortunately, the existence of these smash powers and traces for parametrized spectra depends on technical statements about the bicategory that can be difficult to verify directly, especially if one wants the construction to have a direct geometric interpretation. In this paper, we demonstrate the effectiveness of two tools -- rigidity and deformable functors -- by using them to establish formal structures on this bicategory directly from point-set level data.
\end{abstract}
 
 \setcounter{tocdepth}{1}
 \tableofcontents

\section{Introduction}

The Lefschetz fixed point theorem shows that for a continuous map $f\colon X\to X$, the signed count of the fixed points of $f$ can be given by an algebraic construction, namely the alternating sum of traces on rational homology. Later Dold observed in \cite{dold:index} that that this signed count is also the trace of $\Sigma^\infty_+ f$ in the stable homotopy category. In other words, the fixed points of $f$ may be measured by taking an abstract trace in a symmetric monoidal category \cite{dp}.

This theorem has a generalization due to Wecken \cite{wecken}, which gives a more sensitive fixed-point invariant called the Reidemeister trace $R(f)$. In her thesis, the second author showed that this more sensitive invariant is also a trace, but the trace is carried out in a more sophisticated setting: the bicategory of parametrized spectra \cite{p:thesis}. The result is a version of $R(f)$ that lives in topological Hochschild homology ($\THH$).

In \cite{mp1}, the authors show how to pass to an invariant that counts periodic points in families and lives in topological restriction homology ($\TR$). This invariant is expressed as a trace in the bicategory of \emph{equivariant} parametrized spectra. However, in order to prove that it has the desired behavior, one needs to show that bicategories with shadow have a coherence theorem \cite[Theorem 4.1]{mp1}, and that the bicategory of parametrized spectra has tensor powers, base change objects, and change-of-group functors that interact with the trace in a well-behaved way \cite[Theorem 8.3]{mp1}.

The authors originally sought to give proofs for Theorems 4.1 and 8.3 from \cite{mp1} in a very formal way in any symmetric monoidal bifibration \cite{mp2}. This approach turned out to be very challenging to develop and verify, and so the authors ultimately changed their approach. Theorem 4.1 from \cite{mp1} is now proven in \cite{mp3}, while Theorem 8.3 from \cite{mp1} is now proven in the cases required by \cite{mp1} by Theorems \ref{intro:thm}--\ref{intro:functors} below. The current paper is the last step in proving the main results of \cite{mp1}.

The first part of the structure needed for \cite[Theorem 8.3]{mp1} is called a ``shadowed $n$-Fuller structure.'' In this paper, we use two tools -- rigidity and deformable functors -- and work in the approach to parametrized spectra from \cite{convenient} to give a detailed construction of this structure.
\begin{thm}[\S\ref{sec:shadow_fuller}]\label{intro:thm}
	The bicategory of parametrized spectra $\Ex$ has a shadowed $n$-Fuller structure.
\end{thm}

This extends earlier results from \cite{ms,s:framed,PS:indexed} showing that $\Ex$ is a bicategory with shadow.

We also generalize \cref{intro:thm} in three directions. The first direction involves base change objects in the bicategory $\Ex$ that encode maps of spaces $f\colon A \to B$. Taking the product in $\Ex$ with these base change objects has the effect of taking pullback or pushforward along $f$. We show that these objects can be chosen to be appropriately compatible with the above shadowed $n$-Fuller structure:
\begin{prop}[\S\ref{sec:bc}]\label{intro:bc}
	The bicategory $\Ex$ has a compatible system of base change 1-cells.
\end{prop}
This is related to the notion of a ``framed bicategory'' developed in \cite{s:framed}.

The second direction involves the equivariant and fiberwise versions of the bicategory $\Ex$.
\begin{thm}[\S\ref{sec:GExEx_B}]\label{intro:eq}
	For any finite group $G$ and base space $B$, the bicategories of
	\begin{itemize}
		\item equivariant parametrized spectra $G\Ex$,
		\item fiberwise parametrized spectra $\Ex_B$, and 
		\item equivariant fiberwise parametrized spectra $G\Ex_B$,
	\end{itemize}
	also have shadowed $n$-Fuller structures and compatible systems of base change 1-cells.
\end{thm}

Finally,  we extend \cref{intro:thm} to encompass two change-of-group functors on equivariant parametrized spectra. Let $G$ be a finite group, $H \leq G$ a subgroup, and $WH = NH/H$ the Weyl group of $H$.
\begin{thm}[\S\ref{sec:functors}]\label{intro:functors}
	The geometric fixed points functor and the forgetful functor induce maps of bicategories
	\[ \Phi^H\colon G\Ex \to WH\Ex, \qquad \iota_H^*\colon G\Ex \to H\Ex \]
	and in the fiberwise case
	\[ \Phi^H\colon G\Ex_B \to WH\Ex_B, \qquad \iota_H^*\colon G\Ex_B \to H\Ex_B, \]
	that are compatible with the shadow and the base change objects.
\end{thm}
As already mentioned, all three of these generalizations play important roles in \cite{mp1}.

The rhythm of proof for all of these results is the same.  We first establish a structure at the point-set level, by observing that the functors of interest have no automorphisms other than the identity.  We call such functors {\bf rigid}. Many of the functors we care about on parametrized spectra are rigid by \cite[Thm 4.5.2]{convenient}. We then use this fact to produce coherence isomorphisms and to verify that the relevant compatibilities between coherence isomorphisms hold.

Once we have built the desired structure at the point-set level, we show it descends to the homotopy category using the notion of a {\bf deformable functor} from \cite{dhks}.  This is more general than the notion of a derived functor in a model category, but it has most of the same properties. The extra flexibility is very useful because the most convenient kind of fibration in the parametrized setting does not fit nicely into a model structure. Most importantly, deformable functors allow us to transfer the point-set coherence compatibilities to derived compatibilities in a straightforward way.

While our initial motivation was \cref{intro:thm} and its generalizations, we now regard the primary contribution of this paper to be the  approach to the proofs presented here.  It is much more general than our applications here and has wide ranging uses.

\subsection*{Organization}
In \cref{sec:rigid} we introduce rigid functors and provide the theorem that will allow us to easily verify that the functors of interest in later sections are rigid.  In \cref{sec:point_set} we use the rigidity result of \cref{sec:rigid}  to show that parametrized spectra have a shadowed $n$-Fuller structure and base change objects.  

In \cref{sec:deform} we recall the theory of deformable functors from \cite{dhks}. In \cref{apply_parametrized_spectra} we combine the results of \cref{sec:point_set,sec:deform} to show that the homotopy categories of parametrized spectra assemble to a bicategory with a shadowed $n$-Fuller structure and base change objects.  

\cref{sec:GExEx_B} describes how to modify the results in \cref{sec:point_set,apply_parametrized_spectra} to prove the analogs of \cref{intro:thm} for  equivariant, fiberwise, and equivariant fiberwise parametrized spectra.
In \cref{sec:functors} we apply the techniques of \cref{sec:rigid,sec:deform} to functors between categories of parametrized spectra.

\subsection*{Acknowledgments.}
The first author was partially supported by NSF grants DMS-2005524 and DMS-2052923.
The second author was partially supported by NSF grants DMS-2052905 and DMS-1810779.

\section{Rigid multi-spans}\label{sec:rigid}
In the first three sections of this paper, we work in the symmetric monoidal bifibration (SMBF) of parametrized orthogonal spectra over all base spaces, following the notational conventions of \cite[Thm 6.1.1]{convenient}. This means that we have:
\begin{itemize}
\item a symmetric monoidal category of parametrized orthogonal spectra over all base spaces, denoted $\Osp$  \cite[Def 4.1.1 and 4.1.6]{convenient}, whose monoidal product $\barsmash$ is called the {\bf external smash product},
\item a strict symmetric monoidal functor $\Phi\colon \Osp \to \Top$ that identifies the base space of the parametrized spectrum,
\item for every morphism $f\colon A\to B$ in $\Top$ and spectrum $X$ over $B$, there is a spectrum $f^*(X)$ over $A$ and cartesian arrow $f^*(X)\to X$ over $f$, and 
\item for every morphism $f\colon A\to B$ in $\Top$ and spectrum $Y$ over $A$, there is a spectrum $f_!(Y)$ over $B$ and co-cartesian arrow $f_!(Y)\to Y$ over $f$.
\end{itemize}
We let $\Osp(A)$ denote the fiber of $\Phi$ over $A$. Since $\Phi$ is strict symmetric monoidal, the external smash product of a spectrum $X \in \Osp(A)$ and a spectrum $Y \in \Osp(B)$ gives a spectrum $X \barsmash Y \in \Osp(A \times B)$.

For every map of spaces $f\colon A \to B$, the cartesian arrows define a pullback functor $f^*\colon \Osp(B) \to \Osp(A)$, and the cocartesian arrows define a pushforward functor $f_!\colon \Osp(A) \to \Osp(B)$. The external product preserves cartesian and co-cartesian arrows, which implies that there are canonical isomorphisms $f^*(X \barsmash Y) \cong (f^*X) \barsmash (f^*Y)$ and  $f_!(X \barsmash Y) \cong (f_! X) \barsmash (f_!Y)$.

We also have {\bf Beck-Chevalley isomorphisms} $j_!k^* \cong h^*f_!$ for every pullback square of topological spaces
	\begin{equation}\label{diagonal_square}
	\xymatrix@R=1em @C=1em{
		& A \ar[ld]_-k \ar[rd]^-j & \\
		B \ar[rd]_-f && C \ar[ld]^-h \\
		& D. &
	}
	\end{equation}

This is the structure of a symmetric monoidal bifibration in this particular case. The reader is referred to \cite[12.1]{s:framed} for the definition of a symmetric monoidal bifibration in general. 

In the applications we have in mind, the three operations $\barsmash$, $f^*$, and $f_!$ tend to appear bundled together in the following form.

\begin{defn}\label{df:multi_span}
Given a list of topological spaces $A_1$, $\ldots$, $A_n$, $C$, a {\bf multi-span} from $(A_i)$ to $C$ is a span of the form
\begin{equation}\label{eg:ex_span} \xymatrix @R=5pt { 
		&B \ar[dl]_-f \ar[dr]^-{(g_1,\ldots,g_n)}
		\\
		C && A_1 \times \ldots \times A_n.
	}\end{equation}
In other words, it is a space $B$ and a choice of map from $B$ to each of the spaces $A_1$, $\ldots$, $A_n$, $C$. We say the multi-span in \eqref{eg:ex_span} is {\bf rigid} if the following map is injective:
\[ \xymatrix @R=5pt @C=6em { 
		B \ar[r]^-{(f,g_1,\ldots,g_n)} & C \times A_1 \times \ldots \times A_n.
	}\]
For each multi-span from $(A_i)$ to $C$ as in \eqref{eg:ex_span} we define a functor on parametrized spectra
\begin{equation}\label{action_of_multi_span}
\xymatrix @R=5pt {
	\prod_i \Osp(A_i) \ar[r] & \Osp(C) \\
	(X_1,\ldots,X_n) \ar@{|->}[r] & f_!(g_1,\ldots,g_n)^*(X_1 \barsmash \ldots \barsmash X_n).
}
\end{equation}
We refer to this as the {\bf action} of the multi-span on the tuple of spectra $(X_i)$. 
\end{defn}

\begin{rmk}
This action is why we want to have a product on the right in \eqref{eg:ex_span} but not on the left.  (Of course nothing stops us from replacing $C$ by a product as well.)  
\end{rmk}

\begin{example}\label{ex:no_inputs}
If $n = 0$, the multi-span is of the form $C \overset{f}\leftarrow B \overset{\pi}\rightarrow {\star}$, where $\star$ denotes the one-point space. The action of this multi-span takes no inputs and produces a single parametrized spectrum as output, namely $f_!\pi^*I$, where $I$ is the unit of the symmetric monoidal category $\Osp$. Concretely, $I$ is the sphere spectrum $\Sph$ over the one-point space $\star$, so this action produces $\Sigma^\infty_{+C} B$, the fiberwise suspension spectrum over $C$ of $B_{+C} \cong B \amalg C$. 
\end{example}

Throughout this document we only consider the subcategory $\Osp^c \subseteq \Osp$ of spectra that are freely $f$-cofibrant \cite[Def 4.2.2]{convenient}. The action of a multi-span preserves this class of spectra by \cite[Lem 4.3.1, Thm 4.4.6]{convenient}.

A key property of the functor \eqref{action_of_multi_span} is that for many choices of $g_1,\ldots, g_n,f$ it has very few automorphisms. We say that a functor is {\bf rigid} if the only natural automorphism is the identity transformation. 
\begin{eg}
	The functor
	\[ \xymatrix{
		\Top \times \Top \times \Top \times \Top \ar[r]^-\times & \Top
	} \]
	that sends $(A,B,C,D)$ to the product $A \times B \times C \times D$, is rigid.
	
	The point of rigidity is that it makes it easy to verify coherence diagrams. For instance, once we know the above functor is rigid, we can verify the pentagon identity for the monoidal structure on $\Top$ (compare to \cref{fig:pentagon}) in the following way. The composite of maps around the pentagon is a natural automorphism of $A \times B \times C \times D$. But the only natural isomorphism is the identity, so this composite must be the identity.
\end{eg}

\begin{thm}\label{rigidity}\cite[Thm 4.5.2]{convenient}
	Whenever the multi-span \eqref{eg:ex_span} is rigid, its action \eqref{action_of_multi_span} is a rigid functor.
\end{thm}

As a consequence, any diagram of natural isomorphisms between functors isomorphic to \eqref{action_of_multi_span} automatically commutes.

\begin{rmk}\label{rmk:single_iso}
When $n = 0$, as in \cref{ex:no_inputs}, the rigidity theorem simplifies to the fact that the object $\Sigma^\infty_{+C} B$ has a unique automorphism if $B\to C$ is injective.   This special case will appear frequently in \cref{sec:fuller,sec:base_change}.
\end{rmk}

\section{Structures on the point-set bicategory}\label{sec:point_set}

In this section we use the action of multi-spans from \eqref{action_of_multi_span} to define an extensive formal structure on parametrized spectra. We begin by defining a bicategory of parametrized spectra. Then we add a shadow, compatibility of the external product with symmetry maps, and base change objects for maps of spaces. 
This is the point-set level structure whose associated homotopical structure is used in \cite{mp1} to construct periodic-point invariants.

\subsection{The shadowed bicategory of parametrized spectra}

We start with the simplest and most familiar structure, that parametrized spectra form a bicategory. This result first appears in \cite[17.1.3]{ms}, see also \cite{s:framed}.  The more general result that any symmetric monoidal bifibration (SMBF) defines a bicategory with shadow appears in \cite{PS:indexed}.   A more concise version of the following proof appears in \cite[Thms 6.4.3 and 6.4.8]{convenient}.
We include this proof as a warm-up for 
the more elaborate proofs later in the section.

\begin{prop}\label{lem_bicat}
There is a bicategory of parametrized spectra called $\Osp^c/\Top$, whose 
\begin{itemize}
	\item  0-cells are spaces in $\Top$, and 
	\item 1- and 2-cells from $A$ to $B$ are the category $\Osp^c(A \times B)$ of freely $f$-cofibrant \cite[Def 4.2.2]{convenient} parametrized orthogonal spectra over $A \times B$.
\end{itemize}
\end{prop}

The notion of a bicategory will be recalled in the proof, and uses the following two diagrams. A more extensive treatment of the concept can be found in \cite{benabou,johnson_yau_2cats}.

\begin{figure}[h]
	\begin{subfigure}[t]{18em}
			\xymatrix{
				(X \odot (Y \odot Z)) \odot W \ar[r]^-\alpha & X \odot ((Y \odot Z) \odot W) \ar[d]^-{\id \odot\alpha} \\
				((X \odot Y) \odot Z) \odot W \ar[d]_-\alpha \ar[u]^-{\alpha\odot \id} & X \odot (Y \odot (Z \odot W)) \\
				(X \odot Y) \odot (Z \odot W) \ar@{=}[r] & (X \odot Y) \odot (Z \odot W) \ar[u]_-\alpha
			}
		\caption{Pentagon axiom }\label{fig:pentagon}
	\end{subfigure}
	\hspace{2em}
	\begin{subfigure}[t]{15em}
			\xymatrix{
				(X \odot U_B) \odot Y \ar[d]_-r \ar[r]^-\alpha & X \odot (U_B \odot Y) \ar[d]^-\ell \\
				X \odot Y \ar@{=}[r] & X \odot Y
			}
		\caption{Triangle (unit) axiom }\label{fig:triangle}
	\end{subfigure}
	\caption{Commutative diagrams for a bicategory.}\label{fig:bicat_1}
\end{figure}

Before we begin the proof, as a point of notation, we will use $1$, $\Delta$, and $\pi$ to refer respectively to the identity map, diagonal map, and projection map of a topological space $A$:
\[ 1\colon A \to A, \qquad \Delta\colon A \to A \times A, \qquad \pi\colon A \to \star. \]
We concatenate the names of these maps to refer to products. So for instance
\[ 1\Delta\pi\colon A \times B \times C \to A \times B \times B \]
is the map that sends $(a,b,c)$ to $(a,b,b)$.

\begin{rmk}\label{rmk:suppressed_maps}
In what follows we will often suppress isomorphisms of topological spaces that arise from associator or symmetry maps.  This can be seen above in the lack of  parenthesizations in \eqref{eg:ex_span}.  It will also notably appear in \eqref{eq:m_box_span}. When working with multi-spans, we will only be interested in the induced functors up to isomorphism, and the omitted maps always induce isomorphisms. In some cases we will include parenthesizations to help indicate the type of input spectra.
\end{rmk}

\begin{proof}
For this result we need to produce unit 1-cells $U_A \in \Osp^c(A \times A)$, bicategorical compositions $\odot\colon \Osp^c(A \times B) \times \Osp^c(B \times C) \to \Osp^c(A \times C)$, and associator and unitor isomorphisms that make the diagrams in \cref{fig:bicat_1} commute. We start with the unit and bicategorical composition.

\begin{itemize}
	\item The unit $U_A$ is the suspension spectrum of $A_{+(A \times A)}$ over $A \times A$. Equivalently, it is the unique object produced by the action of the multi-span
	\[ \xymatrix @R=5pt { 
		&A\ar[dl]_-{\Delta}\ar[dr]^-{\pi}
		\\
		A \times A && \prod_\emptyset.
	}\]
	\item For $X \in \Osp^c(A \times B)$ and $Y \in \Osp^c(B \times C)$, the bicategorical product $X \odot Y$ is defined to be the action of the multi-span
	\begin{equation}\label{eq:odot}
	\xymatrix @R=5pt { 
		&A \times B \times C\ar[dl]_-{1\pi 1}\ar[dr]^-{1\Delta 1}
		\\
		A \times C && (A \times B) \times (B \times C).
	}\end{equation}
\end{itemize}
This is the first instance of parenthesization of the spaces along the bottom row by the type of 1-cell input as referenced in \cref{rmk:suppressed_maps}.  In this case we have a 1-cell over $A\times B$ and a 1-cell over $B\times C$. 

We now turn to the associator and unitor maps.  
For these we will make use of the Beck-Chevalley isomorphisms along pullback squares in $\Top$ in \eqref{diagonal_square}.  

\begin{itemize}
	\item To construct the unitor isomorphism $\ell\colon U_A \odot X \cong X$, we notice that the functor $U_A\odot -$ is given by the action of the two multi-spans along the bottom of the diagram below.
	\[ \xymatrix @R=5pt { 
		&&A\times B\ar[dl]_-{\triangle 1}\ar[dr]^-{\triangle1}
		\\
		&A\times A\times B\ar[dl]_-{1\pi1}\ar[dr]^{1\triangle1}&&A\times A\times B\ar[dr]^-{\pi11}\ar[dl]_-{\triangle11}
			\\
	A\times B &&A\times A\times A \times B&&
A\times B
 	}\] 
The commuting square in the middle of this diagram is a pullback square of topological spaces. Therefore we have a Beck--Chevalley isomorphism between the functor $(1\triangle 1)^*(\triangle 11)_!$ and the functor $(\triangle 1)_!(\triangle1)^*$.  Applying this isomorphism and simplifying, we see that the action of the two multi-spans along the bottom is isomorphic to the action of the multi-span along the top, which is just the trivial multi-span of identity maps
	\begin{equation}\label{eq:id_multispan} \xymatrix @R=5pt { 
		&A \times B \ar[dl]_-{11}\ar[dr]^-{11}
		\\
		A \times B && A \times B,
	}\end{equation}
	whose action is the identity functor $\Osp^c(A \times B) \to \Osp^c(A \times B)$. Therefore the functor $X \mapsto U_A \odot X$ and the identity functor $X \mapsto X$ are naturally isomorphic. The multi-span \eqref{eq:id_multispan} is rigid, so by \cref{rigidity}, any two functors that are naturally isomorphic to the action of \eqref{eq:id_multispan}, must have a unique natural isomorphism between them. In particular, the natural isomorphism $U_A \odot X \cong X$ is unique! We can (and must) take the unitor isomorphism $\ell\colon U_A \odot X \cong X$ to be this unique natural isomorphism.
	
	The unitor isomorphism $r\colon X \odot U_B \cong X$ is constructed in the same way.

	\item Similarly, we define the associator $\alpha$ to be the unique natural isomorphism $(X \odot Y) \odot Z \cong X \odot (Y \odot Z)$ of functors $\Osp^c(A \times B) \times \Osp^c(B \times C) \times \Osp^c(C \times D) \to \Osp^c(A \times D)$. To see that this isomorphism exists and is unique, we identify both of the functors $(X \odot Y) \odot Z$ and $X \odot (Y \odot Z)$ with the action of the multi-span
	\[ \xymatrix @R=5pt { 
		&A \times B \times C \times D\ar[dl]_-{1\pi \pi 1}\ar[dr]^-{1\Delta\Delta 1}
		\\
		A \times D && (A \times B) \times (B \times C) \times (C \times D),
	}\]
using Beck-Chevalley isomorphisms to exchange pushforward and pullback functors, just as we did for the unitor isomorphism in the previous stage of the proof. Since the above multi-span is rigid, by \cref{rigidity}, any two functors that are naturally isomorphic to the action of this multi-span must have a unique natural isomorphism between them. So we can (and must) choose $\alpha$ to be this unique natural isomorphism.

Here we have an example of the suppressed isomorphisms referenced in \cref{rmk:suppressed_maps}.  The bottom right product is parenthesized differently for $(X\odot Y)\odot Z$ and $X\odot (Y\odot Z)$, but the resulting spaces are isomorphic, and we omit the map.
\end{itemize}
The axioms for a bicategory now follow similarly:
\begin{itemize}
	\item	For the pentagon axiom in \cref{fig:pentagon}, the multi-span defining the four-fold product $X \odot Y \odot Z \odot W$ is rigid:
	\begin{equation}\label{eq:pentagon} \xymatrix @R=5pt { 
		&A \times B \times C \times D \times E\ar[dl]_(.6){1\pi \pi \pi 1}\ar[dr]^-{1\Delta\Delta\Delta 1}
		\\
		A \times E && (A \times B) \times (B \times C) \times (C \times D) \times (D \times E).
	}\end{equation} All of the operations in \cref{fig:pentagon} are isomorphic to the action of this multi-span. So by \cref{rigidity}, they are uniquely isomorphic to the action of this multi-span, and the diagram in \cref{fig:pentagon} automatically commutes.
	\item The triangle unit axiom in \cref{fig:triangle} follows by the same argument using the following multi-span.
	\begin{equation}\label{eq:triangle_multispan} \xymatrix @R=5pt { 
		&A \times B \times C\ar[dl]_-{1\pi 1}\ar[dr]^-{1\Delta 1}
		\\
		A \times C && (A \times B) \times (B \times C).
	}\end{equation}
\end{itemize}
\end{proof}

The rhythm of this proof will be replicated for the remaining results in this section and for the results in \cref{sec:GExEx_B}.

\begin{defn}\label{df:shadow} Following \cite{p:thesis}, a {\bf shadow} on a bicategory $\mathcal{B}$ is a functor $\sh{\,}\colon \sB(A,A)\to \cat{T}$ to a fixed category $\cat{T}$ for each 0-cell $A$ and natural ``rotator'' isomorphisms 
\[\theta\colon \sh{X\odot Y}\cong \sh{Y\odot X}\]  such that the diagrams in \cref{fig:bicat} commute whenever they are defined. 
\end{defn}

\begin{figure}[h]
	\begin{subfigure}[t]{15em}
			\xymatrix{
				\sh{(X \odot Y) \odot Z} \ar[d]_-\theta \ar[r]^-\alpha & \sh{X \odot (Y \odot Z)} \ar[d]^-\theta \\
				\sh{Z \odot (X \odot Y)} \ar@{<-}[d]_-\alpha & \sh{(Y \odot Z) \odot X} \ar[d]^-\alpha \\
				\sh{(Z \odot X) \odot Y}  & \sh{Y \odot (Z \odot X)} \ar[l]^-\theta
			}
		\caption{Shadow associator axiom }\label{fig:shadow_associator}
	\end{subfigure}
	\hspace{2em}
	\begin{subfigure}[t]{12em}
			\xymatrix{
				\sh{X \odot U_A} \ar[d]_-\theta \ar[rd]^-r & \\
				\sh{U_A \odot X} \ar[d]_-\theta \ar[r]^-\ell & \sh{X} \\
				\sh{X \odot U_A} \ar[ru]^-r &
			}
		\caption{Shadow unitor axiom }\label{fig:shadow_unitor}
	\end{subfigure}
	\caption{Commutative diagrams for a bicategory with shadow.}\label{fig:bicat}
\end{figure}

\begin{prop}\label{df:shadow_span}
The action of the rigid multi-span
\[ \xymatrix @R=5pt { 
		&A\ar[dl]_-{\pi}\ar[dr]^-{\Delta}
		\\
		{\star} && (A \times A).
	}\]
defines a shadow on the bicategory $\Osp^c/\Top$.
\end{prop}

\begin{proof}This proof is similar to the proof of \cref{lem_bicat}  in both structure and details, so we will be more brief.

\begin{itemize}
	\item The rotator isomorphism $\theta\colon \sh{X \odot Y} \cong \sh{Y \odot X}$  exists by comparing both actions to that of the rigid multi-span
	\[ \xymatrix @R=5pt { 
		&A \times B\ar[dl]_-{\pi \pi}\ar[dr]^-{\Delta\Delta}
		\\
		{\star} && (A \times B) \times (B \times A).
	}\] 
\end{itemize}

	This multi-span gives us a more interesting example of the suppressed maps referenced in \cref{rmk:suppressed_maps}.  The map $\Delta\Delta$ has target $A\times A\times B\times B$.  This is replaced by the isomorphic objects $A\times B\times B\times A$ and $B\times A\times A\times B$ to define the source and target of $\theta$.
\begin{itemize}

	\item To prove the associator axiom in \cref{fig:shadow_associator}, we compare all functors to the action of the multi-span
	\[ \xymatrix @R=5pt { 
		&A \times B \times C\ar[dl]_-{\pi \pi \pi}\ar[dr]^-{\Delta\Delta\Delta}
		\\
		{\star} && (A \times B) \times (B \times C) \times (C \times A).
	}\]
	Since this multi-span is rigid, the diagram in \cref{fig:shadow_associator} commutes.
	\item To prove the unitor axiom in \cref{fig:shadow_unitor}, we compare all functors to the action of the rigid multi-span
	\[
	\xymatrix @R=5pt { 
		&A\ar[dl]_-{\pi}\ar[dr]^-{\Delta}
		\\
		{\star} && (A \times A).
	}\]
\end{itemize}
\end{proof}

We have now finished giving the more familiar structure exhibited by the bicategory $\Osp^c/\Top$. For readers primarily interested in the technique of proof, it may make sense to skip ahead to \cref{descend} for the discussion of the descent of this structure to the homotopy category.  The remainder of this section describes the extensive additional structure that the bicategory $\Osp^c/\Top$ enjoys, and which is used in \cite{mp1} to construct periodic-point invariants.

\subsection{Fuller structure}\label{sec:fuller}
In this subsection we describe the compatibility between the external product $\barsmash$, cyclic permutations of the factors, and the shadow on the bicategory  $\Osp^c/\Top$. We will repeatedly use the notion of a \emph{pseudofunctor} (strong functor) and a \emph{pseudonatural transformation}. We will recall what these mean when we need the details, but a more extensive treatment can also be found in \cite{benabou,johnson_yau_2cats}.

For any bicategory $\sB$, a pseudofunctor of the form
\begin{equation}\label{eq:shadowenfuller1}
\boxtimes \colon \underbrace{\sB\times \ldots \times \sB}_n\to \sB
\end{equation}
is given by
\begin{itemize}
	\item a function on 0-cells sending each $n$-tuple $(A_1,\ldots,A_n)$ to a 0-cell that we call $\boxtimes A_i$,
	\item a functor $\boxtimes\colon \prod \sB(A_i,B_i) \to \sB(\boxtimes A_i,\boxtimes B_i)$ for each pair of $n$-tuples of 0-cells $(A_1,\ldots,A_n)$ and $(B_1,\ldots,B_n)$ sending each $n$-tuple of 1-cells $M_i \in \sB(A_i,B_i)$ to a 1-cell $\boxtimes M_i \in \sB(\boxtimes A_i,\boxtimes B_i)$, and
	\item invertible 2-cells that we call the composition and unit isomorphisms
\[ m_\boxtimes\colon (\boxtimes M_i)\odot (\boxtimes N_i) \cong \boxtimes (M_i\odot N_i), \qquad U_{\boxtimes A_i} \cong \boxtimes U_{A_i}, \]
\end{itemize}
such that the diagrams in \cref{fig:shadowed} commute.
\begin{figure}[h]
	\begin{subfigure}[t]{21em}
		\resizebox{\textwidth}{!}{
			\xymatrix{
				(\boxtimes M_i \odot \boxtimes N_i) \odot \boxtimes P_i \ar[r]^-\alpha \ar[d]_-{m_{\boxtimes}\odot \id} &
				\boxtimes M_i \odot (\boxtimes N_i \odot \boxtimes P_i) \ar[d]^-{\id\odot m_\boxtimes} \\
				\boxtimes (M_i \odot N_i) \odot \boxtimes P_i \ar[d]_-{m_{\boxtimes}} &
				\boxtimes M_i \odot \boxtimes (N_i \odot P_i) \ar[d]^-{m_\boxtimes} \\
				\boxtimes ((M_i \odot N_i) \odot P_i) \ar[r]^-\alpha &
				\boxtimes (M_i \odot (N_i \odot P_i))
			}
		}
		\caption{The pseudofunctor associator axiom }\label{fig:pseudo_odot}
	\end{subfigure}

	\begin{subfigure}[t]{17em}
		\resizebox{\textwidth}{!}{
			\xymatrix{
				U_{\boxtimes A_i} \odot (\boxtimes M_i) \ar[d]_-{i_\boxtimes\otimes \id} \ar[r]^-\ell & \boxtimes M_i \\
				(\boxtimes U_{A_i}) \odot (\boxtimes M_i) \ar[r]^-{m_\boxtimes} & \boxtimes (U_{A_i} \odot M_i) \ar[u]_-{\ell}
			}
		}
		\caption{The pseudofunctor unitor axiom for $\ell$}\label{fig:pseudo_unit}
	\end{subfigure}
\hspace{1cm}
	\begin{subfigure}[t]{17em}
		\resizebox{\textwidth}{!}{
			\xymatrix{
				(\boxtimes M_i)\odot U_{\boxtimes B_i}  \ar[d]_-{\id \otimes i_\boxtimes} \ar[r]^-r & \boxtimes M_i \\
				 (\boxtimes M_i)  \odot (\boxtimes U_{B_i})\ar[r]^-{m_\boxtimes} & \boxtimes (M_i\odot U_{B_i} ) \ar[u]_-{r}
			}
		}
		\caption{The pseudofunctor unitor axiom for $r$}\label{fig:pseudo_unit_2}
	\end{subfigure}
	\caption{Commutative diagrams for a pseudofunctor.}\label{fig:shadowed}
\end{figure}

\begin{lem}\label{nfull_funct}
There is a pseudofunctor of bicategories
	\[\boxtimes \colon \underbrace{\Osp^c/\Top \times \ldots \times \Osp^c/\Top}_n\to \Osp^c/\Top \]
that is 
\begin{itemize}
\item the product on 0-cells and 
\item on 
 1-cells and 2-cells is the $n$-fold external smash product \cite[Def 4.4.5]{convenient}, pulled back along the canonical isomorphism 
	\[\prod_i (A_i \times B_i) \cong \prod_i A_i \times \prod_i B_i.\] 
\end{itemize}
\end{lem}

\begin{proof} 
Again this follows the pattern of the proof of \cref{lem_bicat}, so we will be brief. It is enough to produce the composition and unit isomorphisms and to verify the pseudofunctor axioms.
\begin{itemize}
	\item 
	The composition isomorphism $m_\boxtimes \colon (\boxtimes M_i)\odot (\boxtimes N_i) \cong \boxtimes (M_i\odot N_i)$ arises because both sides can be identified with the action of the rigid multi-span
	\begin{equation}\label{eq:m_box_span} \xymatrix @R=5pt { 
		&\prod_i A_i \times \prod_i B_i \times \prod_i C_i\ar[dl]_-{1\pi 1}\ar[dr]^-{1\Delta 1}
		\\
		\prod_i A_i \times \prod_i C_i && \prod_i (A_i \times B_i) \times (B_i \times C_i).
	}\end{equation}
\end{itemize}

	This provides yet another example of the suppressed maps of \cref{rmk:suppressed_maps}.  The input of $(\boxtimes M_i)\odot (\boxtimes N_i)$ is indexed on $\prod_i A_i \times  \prod_i B_i \times  \prod_i B_i \times  \prod_i C_i$ while the result is indexed on $\prod_i A_i \times \prod_i C_i $.  The input of $\boxtimes (M_i\odot N_i)$ is indexed on $\prod_i (A_i \times B_i) \times (B_i \times C_i)$ while the output is indexed on $\prod_i (A_i \times C_i) $. 
\begin{itemize}
	\item 
	Similarly, the unit isomorphism $i_\boxtimes\colon U_{\prod_i A_i} \cong \boxtimes U_{A_i}$ comes from the rigid multi-span
	\[ \xymatrix @R=5pt { 
		&\prod_i A_i \ar[dl]_-{\Delta}\ar[dr]^-{\pi}
		\\
		\prod_i A_i \times \prod_i A_i && \prod_\emptyset.
	}\]
	This is a manifestation of \cref{rmk:single_iso} where 
	the parametrized spectrum has a unique automorphism.
	\item The pseudofunctor associator axiom (\cref{fig:pseudo_unit}) follows using the
rigid multi-span
	\[ \xymatrix @R=5pt { 
		&\prod_i A_i \times \prod_i B_i \times \prod_i C_i \times \prod_i D_i\ar[dl]_(.6){1\pi\pi 1}\ar[dr]^-{1\Delta\Delta 1}
		\\
		\prod_i A_i \times \prod_i D_i && \prod_i (A_i \times B_i) \times (B_i \times C_i) \times (C_i \times D_i)
	}\] 
	\item The pseudofunctor unitor axiom (\cref{fig:pseudo_odot}) follows using the rigid multi-span
	\[ \xymatrix @R=5pt { 
		&\prod_i A_i \times \prod_i B_i\ar[dl]_-{=}\ar[dr]^-{\cong }
		\\
		\prod_i A_i \times \prod_i B_i && \prod_i (A_i \times B_i).
	}\]
\end{itemize}
\end{proof}

Returning to the case of a general bicategory $\sB$ with a pseudofunctor of the form given in \eqref{eq:shadowenfuller1}, let us consider in addition the shift pseudofunctor $\shift\colon \sB\times \ldots \times \sB\to \sB\times \ldots \times \sB$ that permutes the leftmost $\sB$ to the right. So it sends the $n$-tuple of 0-cells $(A_1,\ldots,A_n)$ to the $n$-tuple $(A_2,\ldots,A_n,A_1)$. In shorthand, it sends $(A_i)$ to $(A_{i+1})$. Similarly, it sends every $n$-tuple of 1-cells $(M_i)$ to the $n$-tuple $(M_{i+1})$.

By definition, a pseudonatural transformation
	\begin{equation}\label{eq:shadowenfuller2}
\vartheta\colon \boxtimes \circ \shift \to \boxtimes
\end{equation}
is given by
\begin{itemize}
	\item 1-cells that we call {\bf twisting objects}
	\begin{align*}
	T_{(A_i)} &\in \sB(  \boxtimes A_{i+1}, \boxtimes A_i)
	\end{align*}
	for each $n$-tuple of 0-cells $(A_i)$, and
	\item natural  isomorphisms of functors $\prod \sB(A_i,B_i) \to \sB(\boxtimes A_{i+1},\boxtimes B_i)$
	\begin{equation}\label{eq:vartheta} \vartheta\colon T_{(A_i)}\odot (\boxtimes M_{i}) \xto\cong (\boxtimes M_{i+1})\odot T_{(B_i)},
	\end{equation}
\end{itemize}
such that the diagrams in \cref{fig:shadowed1} commute.
\begin{figure}[h]
	\begin{subfigure}[t]{22em}
		\resizebox{\textwidth}{!}{
			\xymatrix@C=12pt{
				(T_{(A_i)}\odot \boxtimes M_i)\odot \boxtimes N_i\ar[r]^-\alpha \ar[d]_-{\vartheta\odot \id} &
				T_{(A_i)}\odot (\boxtimes M_i\odot \boxtimes N_i)\ar[d]^-{\id \odot m_\boxtimes} \\
				(\boxtimes M_{i+1}\odot T_{(B_i)})\odot  \boxtimes N_i\ar[d]_-\alpha &
				T_{(A_i)}\odot \boxtimes (M_i\odot N_i)\ar[d]^-\vartheta \\
				\boxtimes M_{i+1}\odot (T_{(B_i)}\odot  \boxtimes N_i)\ar[d]_-{\id\odot \vartheta} &
				\boxtimes (M_{i+1}\odot  N_{i+1})\odot T_{(C_i)} \\
				\boxtimes M_{i+1}\odot  (\boxtimes N_{i+1}\odot T_{(C_i)})\ar[r]_-\alpha &
				(\boxtimes M_{i+1}\odot  \boxtimes N_{i+1})\odot T_{(C_i)}\ar[u]_-{m_\boxtimes\odot \id}
			}
		}
		\caption{The pseudonatural composition axiom }\label{fig:shadowed_odot}
	\end{subfigure}
	\begin{subfigure}[t]{18em}
		\resizebox{\textwidth}{!}{
			\xymatrix@C=12pt{
				U_{\prod A_{i+1}} \odot T_{(A_i)} \ar[d]_-{i_\boxtimes\odot \id} \ar[r]^-\ell & T_{(A_i)} & \ar[l]_-r T_{(A_i)} \odot U_{\prod A_i} \ar[d]^-{\id \odot i_\boxtimes} \\
				(\boxtimes U_{A_{i+1}}) \odot T_{(A_i)} \ar[rr]^-\vartheta && T_{(A_i)} \odot (\boxtimes U_{A_i})
			}
		}
		\caption{The pseudonatural unit axiom }\label{fig:shadowed_unit}
	\end{subfigure}
	\caption{Commutative diagrams for a pseudonatural transformation.}\label{fig:shadowed1}
\end{figure}

We will now construct such a pseudonatural transformation on the bicategory $\Osp^c/\Top$.
\begin{itemize}
\item 
	 The action of the multi-span
	\begin{equation}\label{defn_Ta} \xymatrix @R=5pt { 
		&\prod_i A_i\ar[dl]_-{(\shift,1)}\ar[dr]^-{\pi}
		\\
		\prod_i A_{i+1} \times \prod_i A_i && \prod_\emptyset
	}\end{equation}
	defines a 1-cell $T_{(A_i)}$ for each $n$ tuple of objects $(A_1,\ldots ,A_n)$ in $\sB$.
\item 
	 Comparing both sides of the map in \eqref{eq:vartheta} to the action of the multi-span
	\begin{equation}\label{span:vartheta}\xymatrix @R=5pt { 
		&\prod_i A_i \times \prod_i B_i\ar[dl]_-{\shift1}\ar[dr]^-{\cong}
		\\
		\prod_i A_{i+1} \times \prod_i B_i && \prod_i (A_i \times B_i).
	}\end{equation}
gives a unique natural isomorphism $\vartheta\colon T_{(A_i)}\odot (\boxtimes M_{i}) \xto\cong (\boxtimes M_{i+1})\odot T_{(B_i)}$.
\end{itemize}

\begin{lem}\label{nfull_nattran} The 1-cells $T_{(A_i)}$ defined by \eqref{defn_Ta}  and the 2-cells defined by  \eqref{span:vartheta} define a pseudonatural transformation $\vartheta\colon \boxtimes \circ \shift \to \boxtimes$, where  
 $\boxtimes$ is the external smash product pseudofunctor defined in \cref{nfull_funct}.
\end{lem}

\begin{proof}
It is enough to verify that the axioms in \cref{fig:shadowed1} are satisfied for $\vartheta$ and the maps defined earlier in \cref{lem_bicat,nfull_funct}.
\begin{itemize}
\item The pseudonatural composition axiom (\cref{fig:shadowed_odot}) follows from the rigid multi-span 
	\[ \xymatrix @R=5pt { 
		&\prod_i A_i \times \prod_i B_i \times \prod_i C_i\ar[dl]_-{\Gamma\pi 1}\ar[dr]^-{1\Delta 1}
		\\
		\prod_i A_{i+1} \times \prod_i C_i && \prod_i (A_i \times B_i) \times (B_i \times C_i).
	}\]
\item  The pseudonatural unit axiom (\cref{fig:shadowed_unit}) follows from the rigid multi-span 
	\[ \xymatrix @R=5pt { 
		&\prod_i A_i\ar[dl]_-{(\shift,1)}\ar[dr]^-{\pi}
		\\
		\prod_i A_{i+1} \times \prod_i A_i && \prod_\emptyset.
	}\]
\end{itemize}
\end{proof}

Now we are ready to state the main result of this section, a structure on the bicategory of parametrized spectra $\Osp^c/\Top$ that extends the two previous constructions.
\begin{defn}\label{df:shadowed_nfuller}
Fix an integer $n \geq 1$. A {\bf shadowed $n$-Fuller structure} on a bicategory with shadow $\sB$ is the following data and conditions:
\begin{itemize}
	\item a pseudofunctor of bicategories
	$\boxtimes \colon \underbrace{\sB\times \ldots \times \sB}_n\to \sB $ as in \eqref{eq:shadowenfuller1},
	\item 
a pseudonatural transformation
	$\vartheta\colon \boxtimes \circ \shift \to \boxtimes$ as in \eqref{eq:shadowenfuller2}, and
	\item a natural isomorphism 
	\begin{equation}\label{eq:tau} 
\tau\colon \sh{T_{(A_{i-1})}\odot \boxtimes Q_i} \xto\cong \sh{Q_1\odot\ldots\odot Q_n} 
\end{equation}
such that the diagram in  \cref{fig:shadowed_twist} commutes for all $R_i\in \sB(A_{i-1},B_i)$ and $S_i\in \sB(B_i,A_i)$. 
\end{itemize}
\end{defn}

\begin{figure}[h]
		\centerline{
		\xymatrix@R=10pt{\sh{T_{(A_{i-1})}\odots{\prod A_{i-1}} (\boxtimes R_{i}\odot \boxtimes S_i) }\ar[r]^-{\id \odot m_\boxtimes} \ar[d]^\alpha&
			\sh{T_{(A_{i-1})}\odots{\prod A_{i-1}}\boxtimes ( R_{i}\odot S_i) }\ar[r]^-\tau
			&\sh{ R_1\odot  S_1\odot  R_2\odot \ldots \odot R_n\odot  S_n} 
			\ar[ddddd]^\theta
			\\
			\sh{(T_{(A_{i-1})}\odots{\prod A_{i-1}} \boxtimes R_{i})\odot \boxtimes S_i }\ar[d]^{\vartheta\odot \id}
			\\
			\sh{(\boxtimes R_{i+1}\odot T_{(B_{i})})\odots{\prod B_{i}} \boxtimes S_i }\ar[d]^\alpha
			\\
			\sh{\boxtimes R_{i+1}\odot (T_{(B_{i})}\odots{\prod B_{i}} \boxtimes S_i) }\ar[d]^\theta
			\\
			\sh{(T_{(B_{i})}\odots{\prod B_{i}} \boxtimes S_i) \odot\boxtimes R_{i+1} }\ar[d]^\alpha
			\\
			\sh{T_{(B_{i})}\odot (\boxtimes S_i\odots{\prod A_{i}} \boxtimes R_{i+1}) }\ar[r]^-{\id \odot m_\boxtimes} 
			&\sh{T_{(B_{i})} \odot \boxtimes (S_i\odots{\prod A_{i}} R_{i+1})}\ar[r] ^-\tau
			&\sh{ S_1\odot R_2\odot \ldots\odot R_n\odot S_n\odot  R_1}
		}
		}
		\caption{Compatibility with the twist map.}\label{fig:shadowed_twist}
\end{figure}

\begin{prop}\label{prop:n_fuller}
	The bicategory $\Osp^c/\Top$ has a shadowed $n$-Fuller structure with
	\begin{itemize}
		\item the external smash product pseudofunctor $\boxtimes$ from \cref{nfull_funct},
		\item the pseudonatural transformation $\vartheta\colon \boxtimes \circ \shift \to \boxtimes$ from \cref{nfull_nattran}, and 
		\item $\tau$ defined to be the unique natural isomorphism $\sh{T_{(A_{i-1})}\odot \boxtimes Q_i} \xto\cong \sh{Q_1\odot\ldots\odot Q_n}$.
	\end{itemize}
\end{prop}

\begin{proof}
	The isomorphism $\tau$ exists and is unique by comparing both $\sh{T_{(A_{i-1})}\odot \boxtimes Q_i}$ and $\sh{Q_1\odot\ldots\odot Q_n}$ to the action of the rigid multi-span
	\[ \xymatrix @R=5pt { 
		&\prod_i A_i\ar[dl]_-{\pi}\ar[dr]^-{\Delta}
		\\
		{\star} && \prod_i (A_{i-1} \times A_i).
	}\]
	(Note the suppressed isomorphism on the pullback map.)
	The commuting diagram in \cref{fig:shadowed_twist} follows by identifying any one of the functors in the diagram with the action of the rigid multi-span
	\[ \xymatrix @R=5pt { 
		&\prod_i A_i \times \prod_i B_i\ar[dl]_-{\pi\pi}\ar[dr]^-{\Delta\Delta}
		\\
		{\star} && \prod_i (A_{i-1} \times B_i) \times (B_i \times A_i).
	}\]
\end{proof}

\subsection{Base change}\label{sec:base_change}
We next formalize the structure used to define $T_{(A_i)}$ in \eqref{defn_Ta}.

\begin{defn}\label{df_sobc} Let $\bS$ be a 1-category and let $\mathcal{B}$ be a bicategory.  A  {\bf system of base change objects for $\sB$ indexed by $\bS$} is 
a pseudofunctor $[]\colon \bS \to \sB$.\footnote{In \cite{mp1}, this definition and \cref{df_sobc_2} were combined into a single definition.  This had the unfortunate effect of introducing an unsatisfying equality and so we have chosen to separate these definitions here. }
\end{defn}

This is closely related to the notion of a framed bicategory from \cite{s:framed}; any framed bicategory has such a system of base change objects by \cite[Prop 4.5]{s:framed}.

\begin{lem}\label{base_change_functor}\label{prop:base_change_objects}
There is a pseudofunctor $[]\colon \Top \to \Osp^c/\Top$ that is the identity on 0-cells, and that sends each morphism $A\xto{f} B$ to the object of $\Osp^c(B \times A)$ defined by
the action of the multi-span
\begin{equation}\label{span:base_change} \xymatrix @R=5pt { 
	&A\ar[dl]_-{(f,1)}\ar[dr]^-{\pi}
	\\
	B \times A && \prod_\emptyset.
}\end{equation}
\end{lem}

This gives the bicategory $\Osp^c/\Top$ a system of base change objects indexed by the 1-category $\Top$.

Recall from \cref{ex:no_inputs} that a multi-span of this form takes no inputs and produces a single 1-cell as an output. We call this 1-cell $\bcr{A}{f}{B}$, because we think of it as a 1-cell from $B$ to $A$. To extend this to a pseudofunctor, we need to define composition and unit isomorphisms
	\[ m_{[]}\colon \bcr{B}{g}{C} \odot \bcr{A}{f}{B} \cong \bcr{A}{gf}{C},
	\qquad
	i_{[]}\colon U_A \cong \bcr{A}{\id}{A}, \]
that satisfy the diagrams in \cref{fig:base_change_1}.

\begin{figure}[h]
	\begin{subfigure}[t]{32em}
		\resizebox{\textwidth}{!}{
			\xymatrix@C=12pt{
				\left(\bcr{C}{h}{D} \odot \bcr{B}{g}{C}\right) \odot \bcr{A}{f}{B} \ar[r]^-\alpha \ar[d]_-{m_{{[]}}} &
				\bcr{C}{h}{D} \odot \left(\bcr{B}{g}{C} \odot \bcr{A}{f}{B}\right) \ar[d]^-{m_{[]}} \\
				\bcr{B}{hg}{D} \odot \bcr{A}{f}{B} \ar[d]_-{m_{{[]}}} &
				\bcr{C}{h}{D} \odot \bcr{A}{gf}{B} \ar[d]^-{m_{[]}} \\
				\bcr{A}{hgf}{D} \ar@{=}[r] &
				\bcr{A}{hgf}{D}
			}
		}
		\caption{The pseudofunctor associator axiom }\label{fig:bc_odot}
	\end{subfigure}

	\begin{subfigure}[t]{17em}
		\resizebox{\textwidth}{!}{
			\xymatrix@C=16pt{
				U_B \odot \bcr{A}{f}{B} \ar[d]_-{i_{[]}\otimes \id} \ar[r]^-\ell & \bcr{A}{f}{B} \\
				\bcr{B}{\id}{B} \odot \bcr{A}{f}{B} \ar[r]^-{m_{[]}} & \bcr{A}{f}{B} \ar@{=}[u]
			}
		}
		\caption{The pseudofunctor left unitor axiom }\label{fig:bc_unit_l}
	\end{subfigure}
\hspace{1cm}
	\begin{subfigure}[t]{17em}
		\resizebox{\textwidth}{!}{
			\xymatrix@C=16pt{
				\bcr{A}{f}{B} \odot U_B \ar[d]_-{\id\otimes i_{[]}} \ar[r]^-r & \bcr{A}{f}{B} \\
				\bcr{A}{f}{B} \odot \bcr{B}{\id}{B} \ar[r]^-{m_{[]}} & \bcr{A}{f}{B} \ar@{=}[u]
			}
		}
		\caption{The pseudofunctor right unitor axiom }\label{fig:bc_unit_r}
	\end{subfigure}
	\caption{Commutative diagrams for the pseduofunctor in  a system of base change objects.}\label{fig:base_change_1}
\end{figure}

\begin{proof}
	The composition isomorphism $m_{[]}$ is the unique isomorphism $\bcr{B}{g}{C} \odot \bcr{A}{f}{B} \cong \bcr{A}{gf}{C}$ in $\Osp^c(C \times A)$. It exists because both objects can be identified with the output of the action of the rigid multi-span
	\[ \xymatrix @R=5pt { 
		&&A\ar[dl]_-{(gf,f,1)}\ar[dr]^{(\pi,1)}
		\\
		&C\times B\times A\ar[dl]_-{(1,\pi,1)}\ar[dr]^-{(1,\Delta, 1)}&&B\times A\ar[dl]_-{(g,1,f,1)}\ar[dr]^-{\pi}
		\\
		C \times A &&C\times B\times B\times A&& \prod_\emptyset.
	}\]
	Note that the left-hand composite is $(gf,1)$ and the right-hand composite is $\pi$.

	The unit isomorphism $i_{[]}$ is similarly the unique isomorphism $U_A \cong \bcr{A}{=}{A}$. In fact, we defined both of these 1-cells using the same multi-span, so we could even say that $i_{[]}$ is an identity map.

The multi-span in \eqref{span:base_change} is rigid, so base change 1-cells  have no nontrivial automorphisms.   This implies that  the diagrams in \cref{fig:bc_unit_l,fig:bc_unit_r,fig:bc_odot} commute. (This is another manifestation of \cref{rmk:single_iso}.)
\end{proof}

We now describe the compatibility of base change objects with the shadowed $n$-Fuller structure.

\begin{defn}\label{df_iicon}
	Let $\mathcal{C}$ and $\mathcal{D}$ be bicategories, and let $F,G\colon \mathcal{C}\to \mathcal{D}$ be two pseudofunctors such that $F(A)=G(A)$ for every 0-cell $A$ in $\mathcal{C}$. A {\bf vertical natural isomorphism}, also known as an {\bf invertible icon} \cite{lack2010icons}, consists of the following:
\begin{itemize}
\item For each 1-cell $X \in \mathcal{C}(A,B)$, an invertible 2-cell $\eta_X\colon F(X)\to G(X)$. 
\item For each 2-cell $\mu \colon X \to Y$ in $\mathcal{C}(A,B)$, the diagram 
in \cref{icon_nat} commutes. (In other words, $\eta$ defines a natural transformation of functors $\mathcal{C}(A,B) \to \mathcal{D}(F(A),F(B))$.)
\item For each 0-cell $A$ of $\mathcal{C}$,  the diagram of 2-cells
in \cref{icon_2_unit} commutes. 
\item For each composable pair of 1-cells $A\xrightarrow{X}B\xrightarrow{Y}C$ in $\mathcal{C}$,
the diagram of 2-cells in \cref{icon_2_composition}  commutes.
\end{itemize}
\end{defn}

\begin{figure}[h]
	\begin{subfigure}[t]{8em}
\xymatrix{F(X)\ar[r]^-{F(\mu)}\ar[d]^{\eta_{X}}&F(Y)\ar[d]^{\eta_{Y}}\\
G(X)\ar[r]^-{G(\mu)}&G(Y)}
		\caption{}\label{icon_nat}
	\end{subfigure}
\hspace{.5cm}
	\begin{subfigure}[t]{14em}
\xymatrix{F(X)\odot F(Y)\ar[r]^-{\eta_X\odot \eta_Y}\ar[d]^{m_F}&G(X)\odot G(Y)\ar[d]^{m_G}
\\
F(X \odot Y)\ar[r]^-{\eta_{X \odot Y}}&G(X \odot Y)}
		\caption{}\label{icon_2_composition}
	\end{subfigure}
\hspace{.5cm}
	\begin{subfigure}[t]{8em}
\xymatrix{U_{F(A)}\ar@{=}[r]\ar[d]^-{i_F}&U_{G(A)}\ar[d]^-{i_G}
\\F(U_A)\ar[r]^{\eta_{U_A}}&G(U_A)}
		\caption{}\label{icon_2_unit}
	\end{subfigure}
\caption{Commutative diagrams for a vertical natural isomorphism.}\label{fig:base_change_0}
\end{figure}

In our examples, the source bicategory will be a 1-category, which has no nontrivial 2-cells. Therefore we can dispense with \cref{icon_nat} and focus on \cref{icon_2_unit,icon_2_composition}.

\begin{defn}\label{df_sobc_2}  
Let $\sB$ be a bicategory with a system of base change objects  indexed by $\bS$, and suppose that $\bS$ is a cartesian monoidal 1-category.  A {\bf compatible shadowed $n$-Fuller structure} is a shadowed $n$-Fuller structure on $\sB$ (\cref{df:shadowed_nfuller}) where 
\begin{itemize}
\item the twisting object $T_{(B_i)}$ is chosen to be the base change 1-cell $T_{(B_i)}=[\prod B_i\xto{\Gamma} \prod B_{i+1}]$, and 
	\item if  $\prod$ denotes a fixed model for the $n$-fold product in $\bS$, there is  a vertical natural isomorphism $\pbc$ filling the square of pseudofunctors
	\[ \xymatrix{
		\bS^{\times n} \ar[r]^-\prod \ar[d]_-{[]} & \bS \ar[d]^-{[]} \\
		\sB^{\times n} \ar[r]^-\boxtimes & \sB
	} \]
	such that the diagram in  \cref{fig:bc_compatibility} commutes. 

\end{itemize}
\end{defn}

\begin{figure}[h]
\centerline{
		\xymatrix@C=2pt{
		\bcr{\prod B_i}{\cong}{\prod B_{i+1}} \odot \left(\boxtimes \bcr{E_i}{p_i}{B_i}\right) \ar[dd]_-{\id\odot \pbc}^-\cong \ar[r]^-\vartheta_-\cong 
		&
		\left(\boxtimes \bcr{E_{i+1}}{p_{i+1}}{B_{i+1}}\right)\odot \bcr{\prod E_i}{\cong}{\prod E_{i+1}} \ar[d]^-{\pbc\odot\id}_-\cong \\
		&
		\bcr{\prod E_{i+1}}{\prod p_{i+1}}{\prod B_{i+1}} \odot \bcr{\prod E_i}{\cong}{\prod E_{i+1}}\ar[d]_-\cong ^-{m_{[]}}
		\\		
		\bcr{\prod B_i}{\cong}{\prod B_{i+1}} \odot \bcr{\prod E_i}{\prod p_i}{\prod B_i} \ar[r]^-{m_{[]}}_-\cong& \bcr{\prod E_i}{\shift \circ \prod p_i}{\prod B_{i+1}} 
	}
	}
		\caption{The final compatibility axiom.}\label{fig:bc_compatibility}
\end{figure}

\begin{prop}\label{prop:base_change_objects_2}
The base-change objects for the bicategory $\Osp^c/\Top$ from \cref{prop:base_change_objects} are compatible with the shadowed $n$-Fuller structure from \cref{prop:n_fuller} in this sense.
\end{prop}

As in \cref{sec:fuller}, we will prove this result in stages but they are shorter here.

\begin{lem}\label{base_change:iso}There is a 
vertical natural isomorphism $\pbc$ filling the square of pseudofunctors
	\begin{equation}\label{eq:base_change:iso} \xymatrix{
		\Top^{\times n} \ar[r]^-\prod \ar[d]_-{[]} & \Top\ar[d]^-{[]} \\
	(\Osp^c/\Top)^{\times n} \ar[r]^-\boxtimes & \Osp^c/\Top
	} \end{equation}
	where $\prod$ denotes a fixed model for the $n$-fold product in $\Top$. 
\end{lem}

Writing out \cref{df_iicon} in this case, \cref{base_change:iso} is claiming that if we compose the pseudofunctors around either route in the square \eqref{eq:base_change:iso}, we get the same map of 0-cells, and furthermore there are isomorphisms of 1-cells
	\[ \pbc\colon \boxtimes \bcr{A_i}{f_i}{B_i} \cong \bcr{\prod A_i}{\prod f_i}{\prod B_i} \]
that make the diagrams in \cref{fig:base_change} commute. There is no ``naturality'' condition to check for $\pbc$ because $\Top^{\times n}$ is a 1-category, so it has no nontrivial 2-cells.

\begin{figure}[h]
	\begin{subfigure}[t]{28em}
		\resizebox{\textwidth}{!}{
			\xymatrix{
	\bcr{\prod_i B_i}{\prod_i g_i}{\prod_i C_i} \odot \bcr{\prod_i A_i}{\prod_i f_i}{\prod_i B_i} \ar@{<->}[r]^-{\pbc\odot\pbc} \ar[d]^-{m_{[]}}
	&
	\left(\boxtimes_i \bcr{B_i}{g_i}{C_i}\right) \odot \left(\boxtimes_i \bcr{A_i}{f_i}{B_i}\right) \ar[d]^-{m_\boxtimes}
	\\
	\bcr{\prod_i A_i}{\prod_i g_i \circ \prod_i f_i}{\prod_i C_i} \ar@{=}[d]^-{m_{\prod}}
	&
	\boxtimes_i \left(\bcr{B_i}{g_i}{C_i} \odot \bcr{A_i}{f_i}{B_i}\right) \ar[d]^-{m_{[]}}
	\\
	\bcr{\prod_i A_i}{\prod_i (g_i \circ f_i)}{\prod_i C_i} \ar@{<->}[r]^-\pbc
	&
	\boxtimes_i \bcr{A_i}{g_i \circ f_i}{C_i}
}
		}
		\caption{The vertical composition axiom }\label{fig:vert_odot}
	\end{subfigure}
	\begin{subfigure}[t]{14em}
		\resizebox{\textwidth}{!}{
			\xymatrix{
	U_{\prod_i A_i} \ar@{=}[r] \ar@{=}[d]^-{i_{[]}}
	&
	U_{\prod_i A_i} \ar[d]^-{i_\boxtimes}
	\\
	\bcr{\prod_i A_i}{\id}{\prod_i A_i} \ar@{=}[d]^-{i_{\prod}}
	&
	\boxtimes_i U_{A_i} \ar@{=}[d]^-{i_{[]}}
	\\
	\bcr{\prod_i A_i}{\id}{\prod_i A_i} \ar@{<->}[r]^-\pbc
	&
	\boxtimes_i \bcr{A_i}{\id}{A_i}
}
		}
		\caption{The vertical unit axiom }\label{fig:vert_unit}
	\end{subfigure}
	\caption{Commutative diagrams for the vertical natural isomorphism in a system of base change objects.}\label{fig:base_change}
\end{figure}

In \cref{fig:base_change}, all maps that happen to be equalities are notated as such. Some of these maps are not required by \cref{df_sobc} to be equalities, but they are equalities in this construction of the bicategory  $\Osp^c/\Top$.

\begin{proof}

By \cref{nfull_funct}, we know that the two composites of pseudofunctors in \eqref{eq:base_change:iso} agree on 0-cells, taking each $n$-tuple of spaces $(A_i)$ to the product $\prod A_i$. We take $\pbc$ to be the unique isomorphism
	\[ \pbc\colon \boxtimes \bcr{A_i}{f_i}{B_i} \cong \bcr{\prod A_i}{\prod f_i}{\prod B_i} \]
	in $\Osp(\prod B_i, \prod A_i)$. It exists because both sides can be identified with the output of the action of the rigid multi-span
	\[ \xymatrix @R=5pt { 
		&\prod_i A_i\ar[dl]_-{(\prod f_i,1)}\ar[dr]^-{\pi}
		\\
		\prod_i B_i \times \prod_i A_i && \prod_\emptyset.
	}\]

As in the proof of \cref{base_change_functor}, the diagrams of isomorphisms in \cref{fig:vert_unit,fig:vert_odot} must commute because each term in the diagram has no nontrivial automorphisms.
\end{proof}

\begin{lem}\label{base_change:equality}
	Using the equality $T_{(B_i)} = \bcr{\prod B_i}{\cong}{\prod B_{i+1}}$ with the structure defined in \cref{prop:n_fuller,prop:base_change_objects},  the diagram in \cref{fig:bc_compatibility} commutes.
\end{lem}

\begin{proof}
In the proof of \cref{nfull_nattran}, $T_{(B_i)}$ is defined to be the result of the action of the multi-span in \eqref{defn_Ta}.  This is the action of the multi-span in \eqref{span:base_change} when $f$ is taken to be $\shift$.  

As in the proof of \cref{base_change_functor}, the diagram in \cref{fig:bc_compatibility} commutes because each term in the diagram has no nontrivial automorphisms.
\end{proof}

\cref{prop:n_fuller,prop:base_change_objects,base_change:iso,base_change:equality} together imply \cref{prop:base_change_objects_2}.

\section{Deformable functors}
\label{sec:deform}\label{descend}

There are two bicategories of parametrized spectra: a point-set bicategory $\Osp^c/\Top$, in which the 2-cells are actual maps of parametrized spectra, and a homotopy bicategory $\Ex = \Ho\Osp^c/\Top$, in which the 2-cells are morphisms in the homotopy category. To pass from the first to the second, we need to replace the product $\odot$ by a derived product, and similarly for the other operations.

To do this we will use the framework of right-deformable functors from \cite[40.2]{dhks}, see also \cite[\S 3.1]{convenient}. In this section we recall how the framework works in general.

\begin{defn}\label{right_deformable}
	Let $\bC$ and $\bD$ be 1-categories with classes of weak equivalences that satisfy the 2 out of 3 property. A functor $F\colon \cat{C} \to \cat{D}$ is \textbf{right-deformable} if the following data exists:
\begin{itemize}
	\item a \emph{full} subcategory $\cat{A} \subseteq \cat{C}$ 
	such that $F$ preserves weak equivalences on $\cat{A}$,
	\item a  functor $R\colon \cat{C} \to \cat{C}$ landing in $\cat{A}$, and
	\item a natural transformation $\id \to R$ through weak equivalences.   In other words, a natural equivalence $X \simar RX$.
	
\end{itemize}
\end{defn}

In particular, if $\cat C$ is a model category and $F$ is right Quillen, we could take $\cat A$ to be the fibrant objects and $R$ to be a fibrant replacement functor. To emphasize that we will \emph{not} necessarily use fibrant replacement from a model structure, we will instead call the objects of $\cat{A}$ the {\bf \fibrant  objects}, and we say that $R$ is a {\bf \fibrant replacement} functor.

	The \textbf{right-derived functor} of a right-deformable functor $F$ is defined as
	\[ \R F(X) \coloneqq F(RX). \]
The right-derived functor preserves weak equivalences, and therefore defines a map of homotopy categories \[\Ho\cat{C} \to \Ho\cat{D}.\] 

We say ``the'' right-derived functor of $F$, because up to canonical weak equivalence, $\R F$ does not depend on the choice of subcategory $\cat{A}$ or the \fibrant replacement functor $R$. (See \cite[41.1, 41.2]{dhks}, \cite[6.4]{riehl_basic}, \cite[Prop 3.1.1]{convenient}, or \cite[3.5.3]{spectra_book}.) In particular, if $F$ preserves weak equivalences on all of $\cat{C}$, then its right-derived functor is canonically identified with $F$ itself:
\[ \R F \simeq F. \]

\subsection{Natural transformations}
For our applications it is not enough to have individual derived functors.  We need to understand composites of derived functors and natural transformations between them. 

\begin{defn}[{\cite[42.3]{dhks}, \cite[Def 3.2.9]{convenient}}]\label{coherently_right_deformable}
	A list of composable, right-deformable functors
	\begin{equation}\label{composable_functors}
\xymatrix{
	\cat C \ar@{=}[r] & \cat C_0 \ar[r]^-{F_1} & \cat C_1 \ar[r]^-{F_2} & \ldots \ar[r]^-{F_n} & \cat C_n \ar@{=}[r] & \cat D
}
\end{equation}
	is \textbf{coherently right-deformable} if:
\begin{itemize}
\item each functor $F_i$ is right-deformable, meaning the following data exists for each $1 \leq i \leq n$: 
\begin{itemize}
\item a \emph{full} subcategory $\cat A_{i-1} \subseteq \cat C_{i-1}$ such that $F_i$ preserves weak equivalences on $\cat{A}_{i-1}$, 
\item a functor $R_{i-1}\colon \bC_{i-1}\to \bC_{i-1}$ whose image is contained in $\bA_{i-1}$, and 
\item  a natural transformation $\id\to R_{i-1}$ through weak equivalences, 
\end{itemize}
\end{itemize}
and 
\begin{itemize}
\item $F_i(\cat{A}_{i-1})\subset \cat{A}_i$.
\end{itemize}
\end{defn}

We take this approach to derived functors because it gives us good control over the compositions of functors and the natural transformations between them.  The compatibility is described in detail in \cite[Prop 3.2.4, Lem 3.2.10]{convenient}; the relevant statements for us are the following results.
\begin{prop}\label{rd_descends}
	For any list of coherently right-deformable functors, 
	the composite is right-deformable. There is a canonical equivalence of functors (or isomorphism of functors on the homotopy category)
	\[ \R(F_n \circ \cdots \circ F_1) \xrightarrow{\simeq} \R F_n \circ \cdots \circ \R F_1. \]
\end{prop}

\begin{prop}\label{rd_descends_1.5}
	Every natural transformation between right-deformable functors 
\[\eta\colon F \Rightarrow G\] induces a canonical natural transformation of right-derived functors 
\[\ti\eta \colon \R F \Rightarrow \R G\]
 as functors on the homotopy category.
\end{prop}

This canonical natural transformation is characterized as the unique map making the following square commute, as functors $\cat C \to \Ho\cat D$.
\[ \xymatrix{
	F \ar[d] \ar[r]^-\eta & G \ar[d] \\
	\R F \ar@{-->}[r]^-{\ti\eta} & \R G.
} \]
The uniqueness of $\ti\eta$ follows from the universal property of $\R F$. Since $\R F$ is the initial functor under $F$ that preserves equivalences, any map $F \to \R G$ factors uniquely through $\R F$. See e.g. \cite[41.1, 41.2]{dhks}, \cite[6.4]{riehl_basic}, or \cite[3.5.2]{spectra_book}.

As a result, the assignment $\eta \mapsto \ti\eta$ respects vertical compositions of natural transformations, and also respects horizontal compositions when the functors are coherently right-deformable:
\[ \xymatrix @C=3em{
	F \ar[d] \ar@/^1em/[rr]^-{\upsilon \circ \eta} \ar[r]_-\eta & G \ar[r]_-{\upsilon} \ar[d] & H \ar[d] &&
	F_2 \circ F_1 \ar[d] \ar[r]^-{\eta_2 * \eta_1} & G_2 \circ G_1 \ar[d] \\
	\R F \ar@{-->}@/_1em/[rr]_-{(\widetilde{\upsilon \circ \eta}) = \ti\upsilon \circ \ti\eta} \ar@{-->}[r]^-{\ti\eta} & \R G \ar@{-->}[r]^-{\ti\upsilon} & \R H &&
	\R (F_2 \circ F_1) \ar[d]_-\simeq \ar@{-->}[r]^-{(\widetilde{\eta_2 * \eta_1})} & \R (G_2 \circ G_1) \ar[d]^-\simeq \\
	&&&& \R F_2 \circ \R F_1 \ar@{-->}[r]^-{\ti\eta_2 * \ti\eta_1} & \R G_2 \circ \R G_1
} \]
In other words, passage to right-derived functors is a partially-defined 2-functor on categories, functors, and natural transformations.

\begin{cor}\label{rd_descends_2}
	For coherently right-deformable functors, any isomorphism 
\[ G_m \circ \ldots \circ G_1 \cong F_n \circ \ldots \circ F_1\] determines a canonical isomorphism in the homotopy category
\[ 
	\R G_m \circ \ldots \circ \R G_1 \simeq \R F_n \circ \ldots \circ \R F_1. \]
\end{cor}

\begin{cor}\label{rd_descends_3}
	Any commuting diagram between composites of coherently right-deformable functors induces a commuting diagram between the composites of right-derived functors.
\end{cor}

This is easiest to explain with an example. Suppose we have a commuting square of functors of the form
\[ \xymatrix{
	F_2 \circ F_1 \ar[d] \ar[r] & G_3 \circ G_2 \circ G_1 \ar[d] \\
	H_1 \ar[r] & J_2 \circ J_1
} \]
and each of the lists $\{(F_1,F_2),(G_1,G_2,G_3),(H_1),(J_1,J_2)\}$ is coherently right-deformable. Then \cref{rd_descends_3} says this induces a canonical diagram in the homotopy category
\[ \xymatrix{
	\R F_2 \circ \R F_1 \ar[d] \ar[r] & \R G_3 \circ \R G_2 \circ \R G_1 \ar[d] \\
	\R H_1 \ar[r] & \R J_2 \circ \R J_1
} \]
that also commutes. If, furthermore, the map $F_2 \circ F_1 \to G_3 \circ G_2 \circ G_1$ was a composite of two maps $F_2 \to G_3 \circ G_2$ and $F_1 \to G_1$, the corresponding map between right-derived functors would be a composite as well.
	
The upshot is the following. To verify that the functors, natural transformations, and coherences described in \cref{sec:point_set} descend to the homotopy category, it is enough to show that every list of functors that appears is coherently right-deformable. For instance, the expression $(X \odot Y) \odot Z$ corresponds to the list
\[\xymatrix @R=0.5em{
	\Osp^c(A \times B) \times \Osp^c(B \times C) \times \Osp^c(C \times D) \ar[r]^-{\odot \times \id} & \Osp^c(A \times C) \times \Osp^c(C \times D) \ar[r]^-\odot & \Osp^c(A \times D).
} \]
Once we verify that these lists are coherently right-deformable, we get the needed isomorphisms in the homotopy category from \cref{rd_descends_1.5}. The coherence between these isomorphisms follows from \cref{rd_descends_3}.

\begin{rmk}\label{just_put_in_a_ton_of_rs}
The canonical isomorphisms produced by this approach admit an explicit description. If we model composites of right-derived functors by inserting copies of $R_i$ before every functor, then the canonical isomorphisms between composites of right-derived functors are defined by removing all of the intermediate copies of $R_i$ from each composite, keeping the $R_i$s on the initial inputs, applying  the point-set isomorphism, then re-inserting the intermediate copies of $R_i$.  For an example, see \cref{just_put_in_a_ton_of_ps}.
\end{rmk}

\subsection{Extending \fibrant objects}\label{extending_radiance}

In many examples we have some ``default" collection of \fibrant objects $\cat A_i \subseteq \cat C_i$, but we need to extend the collection to a slightly larger full subcategory that will contain the image of $F(\cat A_{i-1})$. In this section we formalize this extension procedure.
\begin{lem}\label{source_of_confusion}
Let $\bA \subseteq \bC$ be a collection of \fibrant objects and $R\colon \bC \to \bC$ a \fibrant replacement for a right-deformable functor $F\colon \bC \to \bD$.

For any collection of objects $\{X_\alpha\}$, $X_\alpha\notin \bA$, if each of the maps
\[F(X_\alpha)\to F(RX_\alpha)\]
 is a weak equivalence, then the full subcategory  $\bA'$ whose objects are the $X_\alpha$ and the objects of $\bA$, is also a collection of \fibrant objects for $F$.
\end{lem}

In this case we say $\bA$ can be {\bf expanded} to contain the objects $\{X_\alpha\}$.

\begin{proof}
It is enough to show that $F$ applied to a weak equivalence between objects in $\bA'$ is a weak equivalence. 
Let $f\colon Y\to Z$ be a weak equivalence, where $Y$ and $Z$ are in $\bA'$.  Then we have a commuting diagram 
\[\xymatrix{F(Y)\ar[r]\ar[d]^{F(f)}&F(RY)\ar[d]^{F(Rf)}
\\
F(Z)\ar[r]&F(RZ)}\]
where the horizontal maps are induced by \fibrant replacement.  The right vertical arrow is a weak equivalence since $RY$ and $RZ$ are in $\bA$.  The top horizontal arrow is a weak equivalence, since either $Y = X_\alpha$ for some $\alpha$ and it is an equivalence by assumption, or $Y \in \bA$ and then it is $F$ of a weak equivalence in $\cat A$. The bottom horizontal is a weak equivalence for the same reason. The left vertical arrow is therefore a weak equivalence by two out of three.
\end{proof}

Our primary functor of interest is $\odot$, which is a functor of two variables $F\colon \cat{C} \times \cat{C}\to \cat{D}$. The most natural collection of \fibrant objects turns out to be a product category $\cat A \times \cat A$ for some $\cat A \subseteq \cat C$, and the \fibrant replacement functor is a product $R \times R$. In this case, we apply \cref{source_of_confusion} by adding \emph{pairs} $(X_\alpha,Y_\alpha) \in \cat C \times \cat C$ to the subcategory $\cat A \times \cat A$, and checking that the functor $R \times R$ induces an equivalence $F(X_\alpha,Y_\alpha) \to F(RX_\alpha,RY_\alpha)$. Often, it is convenient to check this last condition by applying the copies of $R$ one at a time:
\[ \xymatrix{ F(X_\alpha,Y_\alpha) \ar[r]^-\sim & F(RX_\alpha,Y_\alpha) \ar[r]^-\sim & F(RX_\alpha,RY_\alpha). } \]
For instance, this happens in the proof of \cref{unit_preserves_equivalences}.

\section{Structures on the homotopy bicategory $\Ex$}\label{apply_parametrized_spectra}

In this section we prove that  the structure we defined on the point-set bicategory $\Osp^c/\Top$ in \cref{sec:point_set} passes in a canonical way to the homotopy bicategory $\Ex = \Ho\Osp^c/\Top$.

We define this homotopy bicategory as follows. The 0-cells of $\Ex$ are the 0-cells of $\Osp^c/\Top$, in other words, topological spaces. We form the hom categories of $\Ex$ from those of $\Osp^c/\Top$ by inverting the 2-cells in $\Osp^c(A \times B)$ that are stable equivalences in the sense of \cite[Def 5.1.1]{convenient}. These are also called $\pi_*$-isomorphisms in \cite[Def 12.3.4]{ms}.

In other words, the 1-cells of $\Ex$ are the 1-cells of $\Osp^c/\Top$, and the 2-cells of $\Ex$ are zig-zags of 2-cells of $\Osp^c/\Top$, with every backwards-pointing arrow a stable equivalence, up to the usual relations that define a homotopy category, see e.g. \cite[Def 3.1.1]{spectra_book}.

It remains to show that the functors of \cref{sec:point_set} have associated derived functors on $\Ex$ that satisfy the same coherences.

The functors we studied in \cref{sec:point_set} are composites of the external product, pushforward and pullback. We use the following result from \cite[1.0.2]{convenient} to show that they descend to the homotopy category. 
A central role is played by the {\bf level $h$-fibrant spectra}. These are the spectra $X \in \Osp(B)$ for which each of the projection maps $X_n \to B$ is a Hurewicz fibration \cite[Definition 4.2.1]{convenient}.

\begin{thm}\label{convenient_main_thm}
In each category $\Osp^c(B)$ of freely $f$-cofibrant parametrized orthogonal spectra:
\begin{itemize}
	\item the external smash product $\barsmash$ preserves all stable equivalences,
	\item the external smash product of level $h$-fibrant spectra is level $h$-fibrant, 
	\item the pushforward $f_!$ preserves all stable equivalences, and preserves level $h$-fibrant spectra if $f$ is a Hurewicz fibration, and
	\item the pullback $f^*$ preserves level $h$-fibrant spectra and stable equivalences between level $h$-fibrant spectra.
\end{itemize}
	Furthermore, there is a strong symmetric monoidal level $h$-fibrant replacement functor $P$:
	\[ X \simar PX, \qquad P\Sph \cong \Sph, \qquad P(X \barsmash Y) \cong PX \barsmash PY. \]
\end{thm}
More explicitly, $PX$ is defined on each spectrum level $X_n$ by the formula
\[ (PX)_n = (X_n \times_B B^I) \cup_{(B \times_B B^I)} B, \]
where $B^I = \textup{Map}(I,B)$ is the space of paths in $B$ with no conditions on the endpoints, regarded as a space over $B$ by evaluation at $0 \in I$. The map $(PX)_n \to B$ is the evaluation at the other endpoint $1 \in I$.

Since $\barsmash$ and $f_!$ preserve all equivalences, we can take them to be their own right derived functor. The functor $f^*$ preserves equivalences on level $h$-fibrant spectra, so it is right-deformable. We therefore get right-derived functors as in \cite[Rmk 6.1.4]{convenient}:
\[ X \barsmash^{\R} Y = X \barsmash Y, \qquad (\R f_!)(X) = f_!(X), \qquad (\R f^*)(X) = f^*(PX). \]

\begin{rmk}\label{expand_radiant}
\cref{convenient_main_thm} provides our default collection of \fibrant objects for this section: the freely $f$-cofibrant  and level $h$-fibrant spectra. Unfortunately, not all spectra of interest are of this form.  In particular, unit objects and base change objects are not level $h$-fibrant. So, when we need these 1-cells, we use the approach of 
\cref{source_of_confusion} to expand the collection of \fibrant objects. 
\end{rmk}

The following result is the primary tool we will use to extend our collections of \fibrant objects. Let $\cat{Top}(B)$ denote the slice category of topological spaces with a map to $B$. Let $X \in \cat{Top}(A \times B)$ be an arbitrary topological space over $A \times B$. We can add a disjoint copy of $(A \times B)$ to $X$ and take its fiberwise suspension spectrum:
\[ \left( \Sigma^\infty_{+(A \times B)} X \right). \],
The unit $U_A$, twisting object $T_{A_i}$, and, more generally the base change object $\bcr{A}{f}{B}$ are all special cases of this construction.

\begin{lem}\label{convenient:preserve_equivalences}
Let $W,X\in \cat{Top}(A \times B)$, let $Y\in \Osp^c(B \times C)$, and let $f\colon W \to X$ be a weak homotopy equivalence. If the projections $W \to B$ and $X \to B$ are both Serre fibrations, or alternatively if each of the projections $Y_n \to B$ is a Serre fibration, then the map induced by $f$
\[ (\Sigma^\infty_{+(A \times B)} f)\odot \id \colon \left( \Sigma^\infty_{+(A \times B)} W \right) \odot Y \to \left( \Sigma^\infty_{+(A \times B)} X \right) \odot Y \]
is a stable equivalence of parametrized spectra over $A \times C$.
\end{lem}

\begin{proof}
By \cite[Remark 6.4.7]{convenient}, the product $\left( \Sigma^\infty_{+(A \times B)} X \right) \odot Y$ is given at each spectrum level by the space-level circle product $X_{+(A \times B)} \odot Y_n$. By \cite[Cor 6.5.2]{convenient}, this preserves weak homotopy equivalences provided either $X \to B$ or $Y_n \to B$ is a Serre fibration. Therefore under the assumptions in the statement of the lemma, the map 
\[(\Sigma^\infty_{+(A \times B)} f)\odot \id\]
is a weak homotopy equivalence at every spectrum level, so it is a stable equivalence.
\end{proof}

\subsection{The bicategory $\Ex$}  In this section we show that $\Ex$ is a bicategory.  
We first consider the composition product $\odot$.

\begin{lem}\label{odot_deformable}
The functor
\[ \xymatrix{ \odot\colon \Osp^c(A \times B) \times \Osp^c(B \times C) \ar[r] & \Osp^c(A \times C) } \]
preserves stable equivalences between pairs of level $h$-fibrant spectra, and so 
is right-deformable.
\end{lem}

\begin{proof}
This functor is defined as the action of the multi-span \eqref{eq:odot}, so it is composed of an external smash product, a pullback, and a pushforward along a Hurewicz fibration. By \cref{convenient_main_thm}, each of these three steps preserves level $h$-fibrant spectra and equivalences between them, so the first claim follows. For the second claim, we let the subcategory of \fibrant objects be pairs of level $h$-fibrant spectra.
\end{proof}

\begin{lem}\label{associator_collection}
	The lists
\[\xymatrix @R=0.5em{
	\Osp^c(A \times B) \times \Osp^c(B \times C) \times \Osp^c(C \times D) \ar[r]^-{\odot \times \id} & \Osp^c(A \times C) \times \Osp^c(C \times D) \ar[r]^-\odot & \Osp^c(A \times D) \\
	\Osp^c(A \times B) \times \Osp^c(B \times C) \times \Osp^c(C \times D) \ar[r]^-{\id \times \odot} & \Osp^c(A \times B) \times \Osp^c(B \times D) \ar[r]^-\odot & \Osp^c(A \times D)
}\]
are coherently right-deformable (\cref{coherently_right_deformable}).
\end{lem}

\begin{proof}
	At each stage, we use the tuples of level $h$-fibrant spectra as our \fibrant objects. As we observed in the proof of \cref{odot_deformable}, each of the above instances of $\odot$ preserves level $h$-fibrancy, so our subcategories of \fibrant objects are preserved by each of the functors in the list.
\end{proof}

\begin{cor}\label{derived_associator}
	The point-set associator isomorphism induces a canonical associator on the homotopy category
	\[ \ti\alpha\colon (X \odot^\R Y) \odot^\R Z \simeq X \odot^\R (Y \odot^\R Z). \]
\end{cor}

\begin{proof}
	By \cref{associator_collection,rd_descends}, the left-hand side is a right-derived functor of $(X \odot Y) \odot Z$, and the right-hand side is a right-derived functor of $X \odot (Y \odot Z)$. The point-set associator isomorphism therefore induces a canonical isomorphism of derived functors on the homotopy category by \cref{rd_descends_2}.
\end{proof}

\begin{rmk}\label{just_put_in_a_ton_of_ps}
	By \cref{just_put_in_a_ton_of_rs}, this derived associator is given explicitly as the composite
	\[ P(PX \odot PY) \odot PZ \simeq (PX \odot PY) \odot PZ \overset{\alpha}\cong PX \odot (PY \odot PZ) \simeq PX \odot P(PY \odot PZ) \]
where $P$ is the level h-fibrant replacement functor from \cref{convenient_main_thm}.
\end{rmk}

\begin{rmk}
	One might think that we are forced to make $\R\id \times \odot^\R$ the right-derived functor of $\id \times \odot$ here, but that is not the case. We are perfectly free to choose $\id \times \odot^\R$ as the right-derived functor. Any two right-derived functors are canonically identified, so modifying the choice of right-derived functor does not affect the truth of any of the statements we make about it!
\end{rmk}

\begin{lem}\label{pentagon_again}
	The pentagon axiom (\cref{fig:pentagon}) holds in the homotopy category, using the derived product $\odot^\R$ and the derived associator map from \cref{derived_associator}:
\end{lem}
\begin{equation}\label{eq:pentagon_again}
\xymatrix{
	(X \odot^\R (Y \odot^\R Z)) \odot^\R W \ar[r]^-{\ti\alpha} & X \odot^\R ((Y \odot^\R Z) \odot^\R W) \ar[d]^-{\id \odot^\R\ti\alpha} \\
	((X \odot^\R Y) \odot^\R Z) \odot^\R W \ar[d]_-{\ti\alpha} \ar[u]^-{\ti\alpha\odot^\R \id} & X \odot^\R (Y \odot^\R (Z \odot^\R W)) \\
	(X \odot^\R Y) \odot^\R (Z \odot^\R W) \ar@{=}[r] & (X \odot^\R Y) \odot^\R (Z \odot^\R W) \ar[u]_-{\ti\alpha}
}
\end{equation}

\begin{proof}
We first verify the variant of \cref{associator_collection} in which we start with a product of four categories instead of three.  The \fibrant objects are the tuples of  level $h$-fibrant spectra. This shows that every list appearing in \cref{fig:pentagon} is coherently right-deformable.

Since the pentagon axiom (\cref{fig:pentagon}) holds in the point-set bicategory $\Osp^c/\Top$, by \cref{rd_descends_3} we therefore get a commuting diagram between the right-derived functors, giving the commuting diagram \eqref{eq:pentagon_again} in the homotopy category. (The left-hand upper vertical map is the map of right-derived functors induced by $\alpha \odot \id$, and we identify this with $\ti\alpha \odot^\R \id$ by observing that $\ti\alpha \odot^\R \id$ satisfies the defining property from \cref{rd_descends_1.5}. We identify the right-hand upper vertical map in the same way.)
\end{proof}

The remaining diagrams from \cref{sec:point_set} will follow by the same argument. The only thing to check is that every list of functors that arises is coherently right-deformable.  The unit isomorphism requires more care because $U_A$ is not level $h$-fibrant. This is our first use of the ideas in  
\cref{source_of_confusion}.

\begin{lem}\label{unit_preserves_equivalences}
The collection of \fibrant objects for $\odot$ from \cref{odot_deformable} can be expanded to contain pairs $(X,Y)$ where one of $X$ or $Y$ is a unit 1-cell and the other is level $h$-fibrant.  
\end{lem}

\begin{proof}
Recall that $U_A = \Sigma^\infty_{+(A \times A)} A$ and $PU_A = \Sigma^\infty_{+(A \times A)} A^I$, where $A^I = \textup{Map}(I,A)$ is the space of paths in $A$ with no conditions on the endpoints. The three maps
\[ A \xto\id A, \qquad A^I \xto{\textup{ev}_0} A, \qquad \textup{and} \qquad A^I \xto{\textup{ev}_1} A \]
are Hurewicz fibrations. In particular, they are Serre fibrations, so by \cref{convenient:preserve_equivalences} the maps
\[U_A\odot X\to PU_A\odot X \quad 
\text{and} 
\quad X\odot U_A\to X\odot PU_A\] 
are stable equivalences for any parametrized spectrum $X$.

Therefore, if $X$ is level $h$-fibrant, the composites
\[U_A\odot X\to PU_A\odot X \to PU_A \odot PX \quad 
\text{and} 
\quad X\odot U_A\to X\odot PU_A \to PX \odot PU_A \]
are stable equivalences by \cref{odot_deformable}.  By \cref{source_of_confusion}, the subcategory of \fibrant objects can therefore be expanded to include all pairs of the form $(U_A,X)$ and $(X,U_A)$.
\end{proof}

\begin{cor}\label{derived_unitor}
	The point-set unitor isomorphisms induce canonical unitor isomorphisms on the homotopy category
	\[ \ti\ell\colon U_A \odot^\R X \simeq X, \qquad \ti r\colon X \odot^\R U_B \simeq X. \]
\end{cor}

\begin{proof}
	As in the proof of \cref{derived_associator}, by \cref{rd_descends_2} it suffices to show that the two lists
\[\xymatrix @C=3em @R=0.5em{
	{(*)} \times \Osp^c(A \times B) \ar[r]^-{U_A \times \id} & \Osp^c(A \times A) \times \Osp^c(A \times B) \ar[r]^-\odot & \Osp^c(A \times B) \\
	\Osp^c(A \times B) \times {(*)} \ar[r]^-{\id \times U_B} & \Osp^c(A \times B) \times \Osp^c(B \times B) \ar[r]^-\odot & \Osp^c(A \times B)
}\]
are coherently right-deformable. To show this, we take the \fibrant objects in ${(*)} \times \Osp^c(A \times B)$ to be the level $h$-fibrant spectra in $\Osp^c(A \times B)$, and the \fibrant objects in $\Osp^c(A \times A) \times \Osp^c(A \times B)$ to be the pairs $(X,Y)$ in which $X$ is level $h$-fibrant or the unit $U_A$, and $Y$ is level $h$-fibrant. This is a valid collection of radiant objects by \cref{unit_preserves_equivalences}, making the first list coherently right-deformable. The second list follows by the same argument.
\end{proof}

\begin{prop}\label{prop:ex_bicat}\label{ex_is_coherently_deformable}
$\Ex$ is a bicategory with horizontal composition  given by the right derived functor of $\odot$ from \cref{odot_deformable} and $U_A$ as unit 1-cells.
\end{prop}

\begin{proof}
\cref{derived_associator,derived_unitor} give the associativity and unit isomorphisms. We verified the pentagon axiom (\cref{fig:pentagon}) in \cref{pentagon_again}, so it remains to verify the triangle or unit axiom (\cref{fig:triangle}). The three functors that appear in the triangle axiom are as follows:
\begin{equation}
\begin{aligned}
\Osp^c(A \times B) \times \Osp^c(B \times C) & \xto{\odot} \Osp^c(A \times C), \\
\Osp^c(A \times B) \times {(*)} \times \Osp^c(B \times C) &\xto{\id \times U_B \times \id} \Osp^c(A \times B) \times \Osp^c(B \times B) \times \Osp^c(B \times C) \\
	& \xto{\odot \times \id} \Osp^c(A \times B) \times \Osp^c(B \times C) \\
	& \xto{\odot} \Osp^c(A \times C), \textup{ and } \\
\Osp^c(A \times B) \times {(*)} \times \Osp^c(B \times C) &\xto{\id \times U_B \times \id} \Osp^c(A \times B) \times \Osp^c(B \times B) \times \Osp^c(B \times C) \\
	& \xto{\id \times \odot} \Osp^c(A \times B) \times \Osp^c(B \times C) \\
	& \xto{\odot} \Osp^c(A \times C).
\end{aligned}\label{eq:pentagon_derived}
\end{equation}
The first of these is already right-deformable by \cref{odot_deformable}, so we check that the remaining two lists are coherently right-deformable. We take the \fibrant objects to be those spectra that are level $h$-fibrant, except in the category $\Osp^c(B \times B)$, where it can be either the unit $U_B$ or level $h$-fibrant. As in the proof of  \cref{derived_unitor}, this is a valid subcategory of \fibrant objects. The \fibrant objects are preserved  by the functors in \eqref{eq:pentagon_derived} since $(-)\odot U_B$ and $U_B \odot (-)$ are isomorphic to the identity, so they preserve level $h$-fibrant spectra.

By \cref{rd_descends_3}, the triangle axiom (\cref{fig:triangle}) in the point-set bicategory $\Osp^c/\Top$ therefore implies the triangle axiom in the homotopy bicategory $\Ex$.
\end{proof}

\begin{rmk}
In \cite[Theorem 6.4.8]{convenient} several approaches are given for the construction of this bicategory.  This proof is an elaboration on the second approach there. See also \cite[ 17.1.3]{ms} for the original construction by May and Sigurdsson.
\end{rmk}

\subsection{Shadows and Fuller structure}\label{sec:shadow_fuller}

We now turn to the proof that $\Ex$ has a shadowed $n$-Fuller structure. As in the previous subsection, we separate the proof according to how we define the subcategory of \fibrant objects. We start with the cases where the \fibrant objects can be taken to be the level $h$-fibrant spectra and the unit objects $U_A$:

\begin{lem}\label{ex_is_coherently_deformable_2}
The functor $\sh{\,}$ from \cref{df:shadow_span} and the functor $\boxtimes $ from \cref{nfull_funct} are both right-deformable. There are isomorphisms $\ti\theta$ and 
$\ti m_\boxtimes$ making the diagrams in \cref{fig:shadow_associator,fig:shadow_unitor,fig:pseudo_odot} commute for the right derived functors of $\odot$, $\sh{\,}$ and $\boxtimes$.
\end{lem}

\begin{proof}
We take the \fibrant objects for $\sh{\,}$ to be the level $h$-fibrant spectra.
Recall that $\sh{\,}$ is defined as the action of the multi-span in \cref{df:shadow_span}, so it is composed of an external smash product, a pullback, and a pushforward along a Hurewicz fibration. By \cref{convenient_main_thm}, each of these three steps preserves level $h$-fibrant spectra and stable equivalences between them, so $\sh{\,}$ is right-deformable, and, in addition, it preserves the \fibrant objects.

The \fibrant objects for $\boxtimes$ are tuples of level $h$-fibrant spectra.
Recall that $\boxtimes$ is composed of an $n$-fold external smash product and a pullback along an isomorphism, so by \cref{convenient_main_thm} it is also right-deformable and preserves \fibrant objects.

Since the radiant objects for the functors $\odot$, $\sh{\,}$ and $\boxtimes$ are the same and these functors preserve \fibrant objects, any composition of these functors is coherently right-deformable.  So \cref{rd_descends_2,rd_descends_3} allow us to pass from the isomorphisms $\theta$ and $m_\boxtimes$ to isomorphisms of derived functors $\ti\theta$ and $\ti m_\boxtimes$, satisfying the same coherences in the homotopy category that the original isomorphisms satisfied in the point-set category (\cref{fig:shadow_associator,fig:pseudo_odot}).

For \cref{fig:shadow_unitor}, the level $h$-fibrant spectra do not suffice, because they are not preserved by the functors in the following composite.
\begin{equation}
\begin{aligned}
\Osp^c(A \times A) \times {(*)} &\xto{\id \times U_A} \Osp^c(A \times A) \times \Osp^c(A \times A) \\
	& \xto{\odot} \Osp^c(A \times A) \\
	& \xto{\sh{}} \Osp^c(*)
\end{aligned}\label{eq:unit_shadow}
\end{equation}
As in the proof of \cref{prop:ex_bicat}, we fix this by extending the collection of \fibrant objects in $\Osp^c(A \times A) \times \Osp^c(A \times A)$ to include pairs where the first spectrum is level $h$-fibrant and the second spectrum is the unit $U_A$. (We do not include the unit $U_A$ as a \fibrant object in any of the other copies of $\Osp^c(A \times A)$.) This shows that the list in \eqref{eq:unit_shadow}, and each of the other lists appearing in \cref{fig:shadow_unitor}, is coherently right-deformable.
\end{proof}

\begin{rmk}
In the point set category, $\sh{PU_A}$ and $\sh{U_A}$ do not have the same homotopy type -- one gives the suspension spectrum of the free loop space $\Sigma^\infty_+ LA$, while the other gives only the constant loops $\Sigma^\infty_+ A$. So, unlike $\odot$, we cannot expand the \fibrant objects to include the unit $U_A$ in every copy of $\Osp^c(A \times A)$ above.
\end{rmk}

Recall that the collection of \fibrant objects for $\odot$ from \cref{odot_deformable} consists of pairs of level $h$-fibrant spectra.

\begin{lem}\label{base_change_preserves_equivalences}
The collection of \fibrant objects for $\odot$ from \cref{odot_deformable}  can be extended to include pairs where one or both of the spectra in the pair are base change 1-cells  $\bcr{A}{f}{B}$.
\end{lem}

\begin{proof}
This proof is similar to the proof of \cref{unit_preserves_equivalences}. 
We first prove that for any spectrum $Y$  and any level $h$-fibrant spectrum $X$ the following two maps are stable equivalences:
\begin{align}
\label{bccd2}
	\xymatrix{ \bcr{A}{f}{B} \odot Y \ar[r] & P\bcr{A}{f}{B} \odot Y }
\\[-0.3em]
\label{bccd1}
	\xymatrix{ X \odot \bcr{A}{f}{B} \ar[r] & X \odot P\bcr{A}{f}{B} }
\end{align} 
In each case, we tensor with the level $h$-fibrant replacement map.

As in the proof of \cref{unit_preserves_equivalences}, we recall that
\[ \bcr{A}{f}{B} = \Sigma^\infty_{+(B \times A)} A, \qquad P\bcr{A}{f}{B} = \Sigma^\infty_{+(B \times A)} (B^I \times_B A \times_A A^I). \]
The maps $\id\colon A \to A$ and $\ev_1\colon B^I \times_B A \times_A A^I \to A$ are both Serre fibrations, so by \cref{convenient:preserve_equivalences}, the map \eqref{bccd2} is a stable equivalence of parametrized spectra.

On the other hand, if $X \in \Osp^c(C \times B)$ is level $h$-fibrant, then the maps $X_n \to C \times B$ are Hurewicz fibrations, and therefore the projections $X_n \to B$ are Hurewicz fibrations as well. By \cref{convenient:preserve_equivalences}, the map \eqref{bccd1} is therefore a stable equivalence of parametrized spectra.

If $X$ is level $h$-fibrant, the maps
\[ \xymatrix{ X \odot \bcr{A}{f}{B} \ar[r]^-\sim & X \odot P\bcr{A}{f}{B} \ar[r]^-\sim & PX \odot P\bcr{A}{f}{B} } \]
are therefore both stable equivalences, so their composite is a stable equivalence. By \cref{source_of_confusion}, the subcategory of \fibrant objects can therefore be expanded to include all pairs of the form $\left(X,\bcr{A}{f}{B}\right)$.

Similarly, if $Y$ is level $h$-fibrant, the maps
\[ \xymatrix{ \bcr{A}{f}{B} \odot Y \ar[r]^-\sim & P\bcr{A}{f}{B} \odot Y \ar[r]^-\sim & P\bcr{A}{f}{B} \odot PY } \]
are both stable equivalences, so the subcategory of \fibrant objects can be expanded to include all pairs of the form $\left(\bcr{A}{f}{B},Y\right)$ as well.

Finally, if we have two base change 1-cells, we apply \eqref{bccd2} first, then \eqref{bccd1} to get the equivalence
\[ \xymatrix{ \bcr{B}{g}{C} \odot \bcr{A}{f}{B} \ar[r]^-\sim & P\bcr{B}{g}{C} \odot \bcr{A}{f}{B} \ar[r]^-\sim & P\bcr{B}{g}{C} \odot P\bcr{A}{f}{B}. } \]
Therefore the subcategory of \fibrant objects can be expanded to include all pairs of the form $\left(\bcr{B}{g}{C},\bcr{A}{f}{B}\right)$.
\end{proof}

\begin{rmk}
	We will only need \cref{base_change_preserves_equivalences} for pairs in which one spectrum is level $h$-fibrant and the other is a twisting object $T_{(A_i)}$ and for pairs in which both spectra are base change 1-cells.  Since we don't gain much simplification in the proof when considering only these cases, we included the more general result.
\end{rmk}

\begin{rmk}\label{pedantic:expansion} It is important to note that 
we do not think of the extended collection of radiant objects in \cref{base_change_preserves_equivalences} as a further extension of  
the collection from 
\cref{unit_preserves_equivalences}. Instead, we have two different expansions available to us. We are free to use either one at a given stage in a composite of right-deformable functors -- we are not required to use both at the same time.
\end{rmk}

\begin{lem}\label{ex_is_coherently_deformable_3}
For the derived functors of $\odot$, $\sh{\,}$ and $\boxtimes$ defined in \cref{odot_deformable,ex_is_coherently_deformable_2}
there are isomorphisms  
$\ti i_\boxtimes$, $\ti \vartheta$, and $\ti \tau$ satisfying the coherences in \cref{fig:pseudo_unit,fig:shadowed_odot,fig:shadowed_unit,fig:shadowed_twist}.
\end{lem}

\begin{proof}
The proof is essentially the same as that of \cref{pentagon_again,ex_is_coherently_deformable,ex_is_coherently_deformable_2}. To define the isomorphism in the homotopy category
\[ \ti i_{\boxtimes}\colon U_{\prod A_i} \simeq \boxtimes^\R U_{A_i}, \]
the relevant composite of functors is
\[ \xymatrix @C=3em{
(\ast) \ar[r]^-{(U_{A_i})} & \prod \Osp^c(A_i \times A_i) \ar[r]^-{\boxtimes} & \Osp^c((\prod A_i)\times (\prod A_i)).
} \]
The \fibrant objects consist of 
\begin{itemize}
\item the single object of $\ast$ and 
\item the tuples in $\prod \Osp^c(A_i \times A_i)$ in which either 
	\begin{itemize}
	\item every spectrum in the tuple is level $h$-fibrant or 
	\item every spectrum in the tuple is the unit $U_{A_i}$. 
\end{itemize}
\end{itemize}
We do not allow tuples that have some level $h$-fibrant spectra and some units. (Note that it is unnecessary to choose the \fibrant objects in the final category $\Osp^c(\prod A_i \times \prod A_i)$.)
The functor $\boxtimes = \barsmash$ preserves stable equivalences on these objects because it preserves all stable equivalences of freely $f$-cofibrant spectra. 
Then, as usual, \cref{rd_descends_2} defines the isomorphism $\ti i_\boxtimes$.

We now check that each of the lists of functors in \cref{fig:pseudo_unit} is coherently right-deformable.  Since all composites of functors are similar, we illustrate the argument with the composite in the bottom left corner of \cref{fig:pseudo_unit}:
\begin{align*}
(\ast) \times \prod \Osp^c(A_i \times B_i)
&\xto{U_{A_i}\times \id} \prod \Osp^c(A_i \times A_i) \times \prod \Osp^c(A_i \times B_i) \\
&\xto{\boxtimes \times \boxtimes} \Osp^c(\prod A_i \times \prod A_i)\times \Osp^c(\prod A_i \times \prod B_i) \\
&\xto{\odot} \Osp^c(\prod A_i \times B_i).
\end{align*}
Using the \fibrant objects and preservation properties above, we conclude by \cref{rd_descends_3} that the derived version of \cref{fig:pseudo_unit} also commutes.

To define the isomorphism 
\begin{equation}\label{derived_vartheta} \ti\vartheta\colon T_{(A_i)}\odot^\R (\boxtimes^\R M_{i}) \xto\simeq (\boxtimes^\R M_{i+1})\odot^\R T_{(B_i)},
\end{equation}
the relevant composite of functors is
\[ \xymatrix @C=3em{
(\ast) \times \prod \Osp^c(A_i \times B_i) \ar[r]^-{T_{(A_i)}\times \boxtimes} &
\Osp^c(\prod A_{i+1} \times \prod A_i)\times \Osp^c(\prod A_i \times \prod B_i) \ar[r]^-{\odot} &
\Osp^c(\prod A_{i+1}),
} \]
and a similar composite with $T_{(B_i)}$ on the right-hand side. In each of these categories, we take the \fibrant objects to be the tuples of level $h$-fibrant spectra, except in 
\[\Osp^c(\prod A_{i+1} \times \prod A_i),\]
 where we also allow the twisting 1-cell $T_{(A_i)}$. Since $T_{(A_i)}$ is a base change object for a homeomorphism it is, in particular, a base change object for a Serre fibration, so 
\[T_{(A_i)} \odot Y\]
 is level $h$-fibrant when $Y$ is level $h$-fibrant. This allows us to check that \fibrant objects are preserved by the above functors. By \cref{base_change_preserves_equivalences}, the stable equivalences between \fibrant objects are preserved as well, and so the above list of functors is coherently right-deformable. As before, \cref{rd_descends_2} therefore defines the map $\ti\vartheta$.

We similarly check that each of the lists of functors in \cref{fig:shadowed_odot} is coherently right-deformable. We illustrate the argument with the entry in the top left corner of \cref{fig:shadowed_odot}:
\begin{align*}
(\ast) \times \prod \Osp^c(A_i \times &B_i) \times \prod \Osp^c(B_i \times C_i)
\\
&\xto{T_{(A_i)}\times \boxtimes \times \boxtimes} \Osp^c(\prod A_{i+1} \times \prod A_i) \times \Osp^c(\prod A_i \times \prod B_i) \times \Osp^c(\prod B_i \times \prod C_i) \\
&\xto{\odot \times \id} \Osp^c(\prod A_{i+1} \times \prod B_i) \times \Osp^c(\prod B_i \times \prod C_i) \\
&\xto{\odot} \Osp^c(\prod A_{i+1} \times \prod C_i).
\end{align*}
Using the \fibrant objects and preservation properties above, we  conclude by \cref{rd_descends_3} that the derived version of \cref{fig:shadowed_odot} also commutes.

When checking that the lists in \cref{fig:shadowed_unit} are coherently right-deformable, we have to say a little more. We illustrate the argument with the composite in the bottom left corner of \cref{fig:shadowed_unit}:
\begin{align*}
(\ast) \times \prod (\ast)
&\xto{T_{(A_i)}\times \prod( U_{A_i} )} \Osp^c(\prod A_{i+1} \times \prod A_i) \times \prod \Osp^c(A_i \times A_i) \\
&\xto{\id \times \boxtimes} \Osp^c(\prod A_{i+1} \times \prod A_i) \times \Osp^c(\prod A_i \times \prod A_i) \\
&\xto{\odot} \Osp^c(\prod A_{i+1} \times \prod A_i).
\end{align*}
The \fibrant objects are those tuples of spectra in which 
\begin{itemize}
\item each spectrum is level $h$-fibrant, or 
\item every spectrum in the tuple is the ``special'' one -- either 
\begin{itemize}
\item  the twisting 1-cell $T_{(A_i)}$ in the category $\Osp^c(\prod A_{i+1} \times \prod A_i)$, 
\item 
 the unit $U_{A_i}$ in $\Osp^c(A_i \times A_i)$ for \emph{every} value of $i$, or
\item  the product $\boxtimes U_{A_i}$ in $\Osp^c(\prod A_i \times \prod A_i)$. 
\end{itemize}
\end{itemize}
These \fibrant objects are all preserved by construction, and the stable equivalences between them are preserved by \cref{base_change_preserves_equivalences}. We conclude by \cref{rd_descends_3} that the derived version of \cref{fig:shadowed_odot} also commutes.

Finally, to define the isomorphism
	\[ \ti\tau\colon \sh{T_{(A_{i-1})}\odot^\R \boxtimes^\R Q_i}^\R \xto\cong \sh{Q_1\odot^\R\ldots\odot^\R Q_n}^\R, \]  
the relevant composites of functors are
\begin{align*}
(\ast) \times \prod \Osp^c(A_{i-1} \times A_{i})
&\xto{T_{(A_{i-1})}\times \boxtimes} \Osp^c(\prod A_{i} \times \prod A_{i-1}) \times \Osp^c(\prod A_{i-1} \times \prod A_{i}) \\
&\xto{\odot} \Osp^c(\prod A_{i} \times \prod A_{i}) \\
&\xto{\sh{}} \Osp^c(\ast), \\
\prod \Osp^c(A_{i-1} \times A_{i})
&\xto{\odot^{(n-1)}} \Osp^c(A_1 \times A_1) \\
&\xto{\sh{}} \Osp^c(\ast).
\end{align*}
Here the subscripts on $A$ are taken mod $n$. Again, we take the \fibrant objects in each category to be the level $h$-fibrant spectra and the twisting 1-cell $T_{(A_i)}$ in $\Osp^c(\prod A_{i+1} \times \prod A_i)$. These \fibrant objects and the stable equivalences between them are preserved, as before, by \cref{base_change_preserves_equivalences} and the fact that $T_{(A_i)} \odot -$ preserves level $h$-fibrant spectra. We see by a similar choice of \fibrant objects that each list of functors in \cref{fig:shadowed_twist} is coherently right-deformable. We conclude by \cref{rd_descends_3} that the derived version of \cref{fig:shadowed_twist} also commutes.
\end{proof}

Together \cref{ex_is_coherently_deformable_2,ex_is_coherently_deformable_3} show $\Ex$ has a shadowed $n$-Fuller structure. 

\subsection{Base change objects}\label{sec:bc}
It only remains to produce the natural transformations and coherences for a (derived) system of base change objects.

\begin{lem}\label{ex_is_coherently_deformable_4}
For the derived functors of $\odot$, $\sh{\,}$ and $\boxtimes$ defined in \cref{odot_deformable,ex_is_coherently_deformable_2}, 
there are isomorphisms $\ti m_{[]}$, $\ti i_{[]}$,  and $\ti \pbc$ satisfying the coherences in 
\cref{fig:base_change_1,fig:base_change,fig:bc_compatibility}.
\end{lem}

\begin{proof}
To define the isomorphism 
\begin{equation*}
	\ti m_{[]}\colon \bcr{B}{g}{C} \odot^\R \bcr{A}{f}{B} \simeq \bcr{A}{gf}{C},
\end{equation*}
one of the relevant composites of functors is
\[ \xymatrix @C=3em{
(\ast) \times (\ast) \ar[r]^-{[g] \times [f]} &
\Osp^c(C \times B) \times \Osp^c(B \times A) \ar[r]^-{\odot} & \Osp^c(C \times A).
} \]
We take the \fibrant objects to be those pairs $(X,Y) \in \Osp^c(C \times B) \times \Osp^c(B \times A)$ such that either 
\begin{itemize}
\item $X$ and $Y$ are both the base change objects 
or
\item  they are both level $h$-fibrant. 
\end{itemize}
The product $\odot$ preserves stable equivalences on these pairs by \cref{convenient_main_thm,base_change_preserves_equivalences}. The other composite is similar. Therefore this list is coherently right-deformable, so we get the derived isomorphism $\ti m_{[]}$ by \cref{rd_descends_2}.

For the isomorphism
\begin{equation*}
	\ti i_{[]}\colon U_A \simeq \bcr{A}{\id}{A},
\end{equation*}
there is nothing to check, because  a functor of the form $(\ast) \to \Osp^c(A \times A)$ is automatically coherently right-deformable. We therefore get the derived isomorphism $\ti i_{[]}$, which is the same as the previous identification $i_{[]}$ on the point-set level.

We check that each of the lists of functors in \cref{fig:bc_odot,fig:bc_unit_l,fig:bc_unit_r} is coherently right-deformable. We illustrate the argument with the composite in the top left corner of \cref{fig:bc_odot}:
\begin{align*}
(\ast) \times (\ast) \times (\ast)
&\xto{[h] \times [g] \times [f]} \Osp^c(D \times C) \times \Osp^c(C \times B) \times \Osp^c(B \times A) \\
&\xto{\odot \times \id} \Osp^c(D \times B) \times \Osp^c(B \times A) \\
&\xto{\odot} \Osp^c(D \times A).
\end{align*}
We check this is coherently right-deformable by taking the \fibrant objects to be those tuples in which 
\begin{itemize}
\item every spectrum in the tuple is the appropriate base change object or composition of base change objects, or
\item every spectrum in the tuple is level $h$-fibrant. 
\end{itemize}
The \fibrant objects are preserved by construction and the stable equivalences are preserved by \cref{convenient_main_thm,base_change_preserves_equivalences}.
The other lists follow in a similar way. We conclude by \cref{rd_descends_3} that the derived versions of \cref{fig:bc_odot,fig:bc_unit_l,fig:bc_unit_r} also commute.

To define the isomorphism 
\begin{equation*}
	\ti \pbc\colon \boxtimes^\R \bcr{A_i}{f_i}{B_i} \simeq \bcr{\prod A_i}{\prod f_i}{\prod B_i} 
\end{equation*}
one of the relevant composites of functors is
\[ \xymatrix @C=3em{
\prod (\ast) \ar[r]^-{\prod [f_i]} &
\prod \Osp^c(B_i \times A_i) \ar[r]^-{\boxtimes} & \Osp^c(\prod B_i \times \prod A_i).
} \]
We take the \fibrant objects to be those tuples $(X_i) \in \prod \Osp^c(B_i \times A_i)$ such that either 
\begin{itemize}
	\item every $X_i$ is the base change object $\bcr{A_i}{f_i}{B_i}$, or 
	\item every $X_i$ is level $h$-fibrant. 
\end{itemize}
The product $\boxtimes$ preserves stable equivalences on these objects because it preserves all stable equivalences of freely $f$-cofibrant spectra. The other composite is similar.  Therefore this list is coherently right-deformable, so we get the derived isomorphism $\ti \pbc$ by \cref{rd_descends_2}.  

We check that each of the lists of functors in \cref{fig:vert_odot,fig:vert_unit,fig:bc_compatibility} is coherently right-deformable. We illustrate the argument with the composite of functors in the top left corner of \cref{fig:bc_compatibility}:  
\begin{align*}
(\ast) \times \prod (\ast)
&\xto{T_{(B_i)} \times \prod [p_i]} \Osp^c(\prod B_{i+1} \times \prod B_i) \times \prod \Osp^c(B_i \times E_i) \\
&\xto{\id \times \boxtimes} \Osp^c(\prod B_{i+1} \times \prod B_i) \times \Osp^c(\prod B_i \times \prod E_i) \\
&\xto{\odot} \Osp^c(\prod B_{i+1} \times \prod E_i).
\end{align*}
The \fibrant objects are tuples of spectra in which 
\begin{itemize}
\item every spectrum is 
 level $h$-fibrant, or else 
\item every spectrum in the tuple is the ``special'' one -- either 
\begin{itemize}
\item the twisting 1-cell $T_{(B_i)}$ in $\Osp^c(\prod B_{i+1} \times \prod B_i)$, 
\item the base change 1-cells $\bcr{E_i}{p_i}{B_i}$ in $\prod \Osp^c(B_i \times E_i)$, or
\item  the product of base change 1-cells $\boxtimes \bcr{E_i}{p_i}{B_i}$ in $\Osp^c(\prod B_i \times \prod E_i)$.
\end{itemize}
\end{itemize}
Radiant objects are preserved by construction, and stable equivalences between them are preserved by \cref{convenient_main_thm,base_change_preserves_equivalences} (and the fact that the product $\boxtimes \bcr{E_i}{p_i}{B_i}$ is itself isomorphic to a base change 1-cell).

The other lists follow in a similar way. We conclude by \cref{rd_descends_3} that the derived versions of \cref{fig:vert_odot,fig:vert_unit,fig:bc_compatibility} also commute.
\end{proof}

\begin{cor}
	$\Ex$ has a shadowed $n$-Fuller structure and a compatible system of base change objects indexed by $\cat{Top}$.
\end{cor}

\section{The cases of $G\Ex$, $\Ex_B$, and $G\Ex_B$}\label{sec:GExEx_B}

\subsection{Structures on the bicategory $G\Ex$}
Let $G$ be a finite group. Recall from \cite[Thm 7.4.4]{convenient} that there is a $G$-equivariant version of the symmetric monoidal bifibration of parametrized orthogonal spectra, which we denote $G\Osp$, living over the category $G\Top$ of spaces with left $G$-action and $G$-equivariant maps between them. For any $G$-space $B$, the category $G\Osp(B)$ consists of orthogonal $G$-spectra over $B$, and together these have external smash products $\barsmash$, pullbacks $f^*$, and pushforwards $f_!$, with all of the same properties that we described for $\Osp$ in \cref{sec:rigid}. In fact, these are the same constructions as before, with the most obvious $G$-action placed on them. There is also a subcategory of freely $f$-cofibrant spectra $G\Osp^c(B)$ and a notion of level $h$-fibrant spectra, satisfying all of the same properties as in \cref{convenient_main_thm}.

From this we define the point-set bicategory $G\Osp^c/G\Top$ just as before: the 0-cells are $G$-spaces $A$, and the 1-cells and 2-cells are given by the categories $G\Osp^c(A \times B)$ of freely $f$-cofibrant parametrized orthogonal $G$-spectra over the product space $A \times B$. We then define homotopy bicategory $G\Ex$ by inverting the notion of stable equivalence given in \cite[7.4]{convenient}. Informally, this is a genuine notion of stable equivalence, defined using the levels of the spectrum indexed by the orthogonal $G$-representations (\cite[7.4.1]{convenient}).

We therefore have a $G$-equivariant version of the homotopy bicategory of parametrized spectra, denoted $G\Ex$ \cite[7.4.5]{convenient}. We may prove that $G\Ex$ is a bicategory with all of the same structure that we constructed above for $\Ex$. The proofs are the same, with one minor change: in the equivariant version of the rigidity theorem (\cref{rigidity}), the natural automorphism of the action of a multi-span is no longer unique. However, there is still a unique automorphism that is natural with respect to the larger category $G\Osp(B)^{\textup{non}}$ of all parametrized orthogonal $G$-spectra over $B$ and \emph{non-equivariant} maps between them over $B$ \cite[Thm 7.2.4]{convenient}.

All of the isomorphisms that we constructed in \cref{sec:point_set,apply_parametrized_spectra} are indeed natural with respect to the non-equivariant maps of parametrized spectra. Therefore the equivariant version of the rigidity theorem applies to them. This allows us to prove their coherence on the point-set level, and to deduce their coherence on the homotopy category, by exactly the same method as in \cref{sec:point_set,apply_parametrized_spectra}.

\begin{cor}
	$G\Ex$ has a shadowed $n$-Fuller structure and a compatible system of base change objects indexed by $G\Top$.
\end{cor}

\subsection{Structures on the fiberwise point-set bicategory $\Osp^c_{(B)}/\cat{Fib}_B$}

If we work relative to a base space $B$, more modifications are needed. Going back to the beginning, fix a base space $B$ and consider the 1-category $\cat{Fib}_B$, whose objects are Hurewicz fibrations $E \to B$ and whose maps are maps of spaces $E \to E'$ commuting with the projection to $B$. This is a cartesian monoidal category under the fiber product $(D \to B) \times (E \to B) = (D \times_B E \to B)$.

In the category $\cat{Fib}_B$ the diagonal map $\Delta$ and projection map $\pi$ change, since our products are taken relative to $B$:
\[ \Delta\colon A \to A \times_B A, \qquad \pi\colon A \to B. \]
When we restrict to the fiber over any one point of $B$, these become the same maps from before.

Consider the symmetric monoidal bifibration $\Osp_{(B)}$ of all parametrized orthogonal spectra over all spaces over $B$, defined in \cite[Def 6.3.1]{convenient}. As before, this consists of the following data.
\begin{itemize}
	\item A symmetric monoidal category $\Osp_{(B)}$ of all parametrized orthogonal spectra $X$ over base spaces $E$ with Hurewicz fibrations $E \to B$. The maps in this category are the same as in $\Osp$.  The monoidal product is different however: it is the external smash product rel $B$, denoted $\barsmash_B$. For a spectrum $X$ over $D$ and $Y$ over $E$, this product $X \barsmash_B Y$ is defined by pulling back $X \barsmash Y$ along the canonical map $D \times_B E \to D \times E$.
	\item A projection functor $\Phi\colon \Osp_{(B)} \to \cat{Fib}_B$ that remembers only the base space, and is strict symmetric monoidal.
	\item For every morphism $f\colon D \to E$ in $\cat{Fib}_B$, a pullback functor $f^*$ and $f_!$, which is in fact the same functor as before from $\Osp$, ignoring the projection map of the base space to $B$.
\end{itemize}
Since this new structure on $\Osp_{(B)}$ is defined in terms of the old one on $\Osp$, we get the same compatibilities that make $\Osp_{(B)}$ a symmetric monoidal bifibration: canonical isomorphisms $f^*(X \barsmash_B Y) \cong (f^*X) \barsmash_B (f^*Y)$ and  $f_!(X \barsmash_B Y) \cong (f_! X) \barsmash_B (f_!Y)$, and Beck-Chevalley isomorphisms $j_!k^* \cong h^*f_!$ for every pullback square of topological spaces over $B$ as in \eqref{diagonal_square}.\footnote{The fact that $f_!$ commutes with $\barsmash_B$ is subtle and uses the fact that the projection maps $E \to B$ are Hurewicz fibrations. See \cite[Prop 6.3.2]{convenient}.}

Now consider multi-spans in the category $\cat{Fib}_B$. In other words, all of the spaces have a Hurewicz fibration to $B$ and all of the maps are over $B$:
\begin{equation}\label{eq:multi_span_B} \xymatrix @R=5pt @!C=4em { 
		&E \ar[dl]_-f \ar[dr]^-{(g_1,\ldots,g_n)} \ar[dd]
		\\
		C \ar[rd] && A_1 \times_B \ldots \times_B A_n \ar[ld] \\
		& B
	}\end{equation}
As in \cref{df:multi_span}, we say that a multi-span in $\cat{Fib}_B$ is {\bf rigid} if the map $(f,g_1,\ldots,g_n)$ is injective. It doesn't matter whether we think of this as a map to the product $C \times A_1 \times \ldots \times A_n$ or a map to the fiber product $C \times_B A_1 \times_B \ldots \times_B A_n$ when measuring whether it is injective.

	The multi-span in \eqref{eq:multi_span_B} acts on the categories of orthogonal spectra as before, but now using the external smash product rel $B$:
\begin{equation}\label{action_of_multi_span_fiberwise}
\xymatrix @R=5pt {
	\prod_i \Osp^c_{(B)}(A_i) = \prod_i \Osp^c(A_i) \ar[r] & \Osp^c_{(B)}(C) = \Osp^c(C) \\
	(X_1,\ldots,X_n) \ar@{|->}[r] & f_!(g_1,\ldots,g_n)^*(X_1 \barsmash_B \ldots \barsmash_B X_n).
}
\end{equation}
Since $\barsmash_B$ is defined by composing $\barsmash$ with a pullback, this action is still composed of a smash product, a pullback, and a pushforward. Therefore it still preserves freely $f$-cofibrant spectra, and it follows from the rigidity theorem (\cref{rigidity}) that:
\begin{cor}\label{rigidity_2}
	If the multi-span \eqref{eq:multi_span_B} is rigid, then its action as defined in \eqref{action_of_multi_span_fiberwise} is rigid.
\end{cor}

We define the point-set bicategory $\Osp^c_{(B)}/\cat{Fib}_B$ as before:
\begin{itemize}
	\item The 0-cells are objects of $\cat{Fib}_B$, so they are Hurewicz fibrations $E$ over $B$.
	\item The 1- and 2-cells from $A$ to $C$ are the category $\Osp^c(A \times_B C)$ of freely $f$-cofibrant parametrized orthogonal spectra over the fiber product $A \times_B C$.
	\item The unit $U_A$ is the suspension spectrum of $A_{+(A \times_B A)}$ over $A \times_B A$. Equivalently, it is unique object produced by the action of the multi-span
	\[ \xymatrix @R=5pt { 
		&A\ar[dl]_-{\Delta}\ar[dr]^-{\pi}
		\\
		A \times_B A && B.
	}\]
	\item For $X \in \Osp^c(A \times_B E)$ and $Y \in \Osp^c(E \times_B C)$, the product $X \odot_B Y$ is the action of the multi-span
	\[ \xymatrix @R=5pt { 
		&A \times_B E \times_B C\ar[dl]_-{1\pi 1}\ar[dr]^-{1\Delta 1}
		\\
		A \times_B C && (A \times_B E) \times_B (E \times_B C).
	}\]
	\item The associator $\alpha$ is the unique natural isomorphism $(X \odot_B Y) \odot_B Z \cong X \odot_B (Y \odot_B Z)$. It exists by identifying both sides with the action of the multi-span
	\[ \xymatrix @R=5pt { 
		&A \times_B E \times_B C \times_B D\ar[dl]_-{1\pi \pi 1}\ar[dr]^-{1\Delta\Delta 1}
		\\
		A \times_B D && (A \times_B E) \times_B (E \times_B C) \times_B (C \times_B D).
	}\]
	Since this multi-span is rigid, by \cref{rigidity_2}, any two functors isomorphic to this action have a unique natural isomorphism between them.
	\item Similarly, the unitors are the unique isomorphisms $U_A \odot_B X \cong X \cong X \odot_B U_C$.
	\item The shadow $\sh{X}$ is defined to be the action of the multi-span
\[ \xymatrix @R=5pt { 
		&A\ar[dl]_-{\pi}\ar[dr]^-{\Delta}
		\\
		{B} && (A \times_B A).
	}\]
	\item The rotator isomorphism $\theta\colon \sh{X \odot_B Y} \cong \sh{Y \odot_B X}$ exists by comparing both actions to that of the rigid multi-span
	\[ \xymatrix @R=5pt { 
		&A \times_B E\ar[dl]_-{\pi \pi}\ar[dr]^-{\Delta\Delta}
		\\
		{B} && (A \times_B E) \times_B (E \times_B A).
	}\]
\end{itemize}

\begin{rmk}
Note that the object $U_A$ is not the same as the object we denote using this symbol in \cref{sec:point_set,apply_parametrized_spectra}, but we will rely on context to disambiguate.   We will have a similar ambiguity for the base change objects $T_{(A_i)}$ and we will also rely on context to distinguish.
\end{rmk}

\begin{prop}
	The above functors and natural isomorphisms make $\Osp^c_{(B)}/\cat{Fib}_B$ into a bicategory with shadow. Furthermore it has a shadowed $n$-Fuller structure in which $\boxtimes = \barsmash_B$, and base change objects indexed by the category $\cat{Fib}_B$.
\end{prop}

\begin{proof}
	The arguments are the same as those in \cref{sec:point_set}, where every product is now a fiber product over $B$.
\end{proof}

\subsection{Structures on the bicategory $\Ex_B$}

Our next goal is to invert the stable equivalences in each of the categories $\Osp^c_{(B)}(A \times_B C) = \Osp^c(A \times_B C)$, giving the fiberwise bicategory of parametrized spectra $\Ex_B$, and prove that all of these structures pass to $\Ex_B$.

Before we proceed, we need two lemmas. The first is a fiberwise version of \cref{convenient:preserve_equivalences}. Let $X \in \cat{Top}(A \times_B C)$ be an arbitrary topological space over $A \times_B C$. We can add a disjoint copy of $(A \times_B C)$ to $X$ and take its fiberwise suspension spectrum:
\[ \left( \Sigma^\infty_{+(A \times_B C)} X \right). \]
As before, the unit $U_A$, twisting object $T_{A_i}$, and base change object $\bcr{A}{f}{C}_B$ are all special cases of this construction.

\begin{lem}\label{convenient:preserve_equivalences_fiberwise}
Let $W,X\in \cat{Top}(A \times_B C)$, let $Y\in \Osp^c(C \times_B D)$, and let $f\colon W \to X$ be a weak homotopy equivalence. If the projections $W \to C$ and $X \to C$ are both Serre fibrations, or alternatively if each of the projections $Y_n \to C$ is a Serre fibration, then the map induced by $f$
\[ (\Sigma^\infty_{+(A \times_B C)} f)\odot \id \colon \left( \Sigma^\infty_{+(A \times_B C)} W \right) \odot_B Y \to \left( \Sigma^\infty_{+(A \times_B C)} X \right) \odot Y \]
is a stable equivalence of parametrized spectra over $A \times_B D$.
\end{lem}

\begin{proof}
As in \cref{convenient:preserve_equivalences}, the product $\left( \Sigma^\infty_{+(A \times_B C)} X \right) \odot_B Y$ is given at each spectrum level by the space-level circle product $X_{+(A \times_B C)} \odot_B Y_n$. By \cite[Lem 6.5.1]{convenient}, this preserves weak homotopy equivalences provided either $X \to C$ or $Y_n \to C$ is a Serre fibration, which is enough to conclude that $(\Sigma^\infty_{+(A \times_B C)} f)\odot \id$ is a stable equivalence.
\end{proof}

The next lemma helps us check the fibration condition from \cref{convenient:preserve_equivalences_fiberwise}.

\begin{lem}\label{weird_fibration}
	If $f\colon D \to E$ is any map of Hurewicz fibrations over $B$, the composite
\begin{equation}\label{parametrized_fibration} D \times_{(E \times_B D)} (E \times_B D)^I \xto{\pi 1} (E \times_B D)^I \xto{(\pi1)^I} D^I \xto{\textup{ev}_1} D \end{equation}
	is a Hurewicz fibration and homotopy equivalence.
\end{lem}

The product $D \times_{(E \times_B D)} (E \times_B D)^I$ is formed using evaluation at $0 \in I$, while the composite in \eqref{parametrized_fibration} uses evaluation at $1 \in I$.  The map $D\to E\times_BD$ is the graph of $f$.

\begin{proof}
A homotopy inverse to the map in \eqref{parametrized_fibration} is the map $d\mapsto (d,c_{f(d),d})$ where $c_{f(d),d}$ is the constant path at $(f(d),d)$.  The endomorphism of $D$ given by these maps is the identity.  The endomorphism of $D \times_{(E \times_B D)} (E \times_B D)^I $ given by these maps is homotopic to the identity by shrinking the path in $E\times_BD$.

To see that \eqref{parametrized_fibration} is a Hurewicz fibration, we first rewrite it as the composite
\begin{equation}\label{parametrized_fibration_2} D \times_{(E \times_B D)} (E \times_B D)^I \xto{\pi1} (E \times_B D)^I \xto{\textup{ev}_1} E \times_B D \xto{\pi1} D. \end{equation}
Now suppose we are given the solid arrows in the following commutative diagram:
\begin{equation}\label{fib:lifting_square}\xymatrix{X\ar[r]^-{f_1\times f_2}\ar[d]_{i_0}
	&D \times_{(E \times_B D)} (E \times_B D)^I\ar[d]^{\eqref{parametrized_fibration_2}}
	\\
	X\times I \ar[r]^-H\ar@{.>}[ur]^-{K_1\times K_2}
	&D.
}\end{equation}
To construct the dotted lift, we let $p\colon X \times I \to X$ be the projection, and define $K_1 := f_1 \circ p$. Then it suffices to find a lift $K_2$ in the following diagram, where $G\colon X \times I \to E \times_B D$ is the map defined by $(f \times \id) \circ f_1 \circ p$, or equivalently $\ev_0 \circ f_2 \circ p$.
\begin{equation}\label{fib:lifting_square_expanded}
\xymatrix @C=6em{X\ar[r]^-{f_2}\ar[dd]^{i_0}&(E\times_BD)^I\ar[d]^{(\ev_0,\ev_1)}
\\
&(E\times_BD) \times (E\times_BD)\ar[d]^-{11\pi 1}
\\
X\times I \ar@{.>}[uur]^{K_2}\ar[r]^-{(G,H)} &(E\times_BD) \times D
}
\end{equation}
The product of the two evaluation maps $(\ev_0,\ev_1)$ is a Hurewicz fibration since it is the dual of the inclusion $\partial I \to I$. The projection $\pi 1\colon E\times_BD\to D$ is also a Hurewicz fibration since it is a pullback of $E\to B$.  
Therefore the lift $K_2$ exists. 
\end{proof}

As in \cref{apply_parametrized_spectra}, the main thing we need to pass to the homotopy category is to select collections of \fibrant objects that make each list of functors coherently deformable.

As in \cref{odot_deformable}, the pairs of level $h$-fibrant spectra are a collection of \fibrant objects for the functor $\odot_B$.  (For a spectrum in $\Osp^c(A \times_B C)$, these will be level $h$-fibrant with respect to the base space $A\times_BC$.)
\begin{lem}\label{fib_source_of_confusion}
As in \cref{unit_preserves_equivalences,base_change_preserves_equivalences},
this collection of \fibrant objects for $\odot_B$ can be extended to contain either
\begin{itemize}
\item pairs in which one spectrum is a unit $U_A$, 
\item pairs in which one spectrum is a twisting 1-cell $T_{(A_i)}$, or 
\item pairs in which both spectra are base change 1-cells $\bcr{A}{f}{C}_B$.
\end{itemize}
\end{lem}

As in \cref{pedantic:expansion}, we do not need to extend the collection of \fibrant objects to contain all of these simultaneously. 

\begin{proof}
For the first case where one spectrum is a unit $U_A$, as in the proof of \cref{unit_preserves_equivalences}, it is enough to check that the map
\begin{align}
	U_A \odot_B X &\xto\sim  (PU_A) \odot_B X \label{fibrant_unit_fiberwise_1}
\end{align}
is a stable equivalence. Recall that in the fiberwise setting over $B$, we have
\[ U_A = \Sigma^\infty_{+(A \times_B A)} A, \qquad PU_A = \Sigma^\infty_{+(A \times_B A)} (A \times_{(A \times_B A)} (A \times_B A)^I). \]
By \cref{convenient:preserve_equivalences_fiberwise}, it is enough to check that the following two maps are Serre fibrations:
\begin{align}
	A &\xto\id A,\label{eq:complicated_0}
	\\
	A \times_{(A \times_B A)} (A \times_B A)^I \to (A \times_B A)^I &\xto{\pi 1} A^I \xto{\textup{ev}_1} A. \label{eq:complicated_1}
\end{align}
But the first is an identity map, and the second is the special case of \cref{weird_fibration} for the identity map $\id\colon A\to A$. Therefore they are both Hurewicz fibrations, so \eqref{fibrant_unit_fiberwise_1} is a stable equivalence.

For the second case where one spectrum is a twisting 1-cell $T_{(A_i)}$, we follow the same proof as the unit case, rather than following the proof of \cref{base_change_preserves_equivalences}. It suffices to prove that the maps
\begin{align}
 (\prod_B A_i) &\xto\id (\prod_B A_i),\label{eq:complicated_3}
\\ 
(\prod_B A_i) \times_{(\prod_B A_{i+1} \times_B \prod_B A_i)} (\prod_B A_{i+1} \times_B \prod_B A_i)^I &\to (\prod_B A_{i+1} \times_B \prod_B A_i)^I \xto{\pi 1} (\prod_B A_i)^I \xto{\textup{ev}_1} (\prod_B A_i) \label{eq:complicated_2}
\end{align}
are Serre fibrations. This follows once again as a special case of \cref{weird_fibration}.

Finally, for the case of two base change objects it is enough to check that the map
\begin{align*}
	\bcr{C}{g}{D}_B \odot_B \bcr{A}{f}{C}_B &\xto{\sim}  P\bcr{C}{g}{D}_B \odot_B P\bcr{A}{f}{C}_B\label{fibrant_unit_fiberwise_6}
\end{align*}
is a stable equivalence of parametrized spectra. Using \cite[Remark 6.4.7]{convenient} and the fact that $P$ commutes with suspension spectra, this is a map of suspension spectra. Therefore, it suffices to focus attention at spectrum level zero and ignore the disjoint basepoint sections. This gives the inclusion of constant paths map 
\begin{equation}\label{big_ol_constant_paths}
A \xto\sim C \times_{(D \times_B C)} (D \times_B C)^I \times_C (C \times_B A)^I \times_{(C \times_B A)} A.
\end{equation}
To see this is a homotopy equivalence, 
consider the composite
\begin{align*}
C \times_{(D \times_B C)} (D \times_B C)^I \times_C (C \times_B A)^I \times_{(C \times_B A)} A&\to C \times_C (C \times_B A)^I \times_{(C \times_B A)} A
\\
&\cong (C \times_B A)^I \times_{(C \times_B A)} A
\\
&\to A
\end{align*}
where the first piece applies the map in \cref{weird_fibration} to the part to the left of the $\times_C$, the second removes the redundant $C$, and the third applies the ``reflection'' of the map in \cref{weird_fibration}.  These three maps are homotopy equivalences, and their composition is a left inverse to \eqref{big_ol_constant_paths}. Therefore \eqref{big_ol_constant_paths} is a homotopy equivalence as well.
\end{proof}

Now we define $\Ex_B$ by taking the point-set bicategory $\Osp^c_{(B)}/\cat{Fib}_B$ and inverting the stable equivalences in each category $\Osp^c(A \times_B C)$.

\begin{prop}\label{ex_b_is_coherently_deformable}
	Each functor in the definition of the shadowed $n$-Fuller structure  and base change objects for $\Osp^c_{(B)}/\cat{Fib}_B$ has a right-derived functor. For every isomorphism and coherence axiom in these structures, the relevant lists of functors are coherently right-deformable.
\end{prop}

\begin{proof}
This proof is almost identical to that of \cref{ex_is_coherently_deformable,ex_is_coherently_deformable_2,ex_is_coherently_deformable_3,ex_is_coherently_deformable_4}. 
The role of \cref{unit_preserves_equivalences,base_change_preserves_equivalences} is played by \cref{fib_source_of_confusion}.

The only significant difference is that $\boxtimes$ is replaced by $\barsmash_B$, which is the external smash product followed by a pullback along a diagonal map $B \to B^n$.  This does not preserve all stable equivalences of freely $f$-cofibrant spectra over $A$, only those for which the composite maps $X_n \to A \to B$ are Serre fibrations. So we have to check that in each case it is applied, each of our \fibrant objects has this property.

For the level $h$-fibrant spectra $X$ over $A \to B$, the map $X_n \to A$ is a Hurewicz fibration.  Since $A \to B$ is a Hurewicz fibration by assumption, the composite $X_n \to B$ is a Hurewicz fibration. The only other cases are when $X$ is either a unit or a base change object. Since units are a special case of base change objects, we only have to consider base change objects.

If $X = \bcr{A}{f}{C} = \Sigma^\infty_{+(C \times_B A)} A$, then at spectrum level 0, the projection to $B$ is
\[ A_{+(C \times_B A)} \to C \times_B A \to B, \]
which is a Hurewicz fibration because $A \to B$ and $C \to B$ are Hurewicz fibrations by assumption. At spectrum level $n$, we get the external smash product
\[ S^n \barsmash \left( A_{+(C \times_B A)} \right). \]
By \cite[(2.3.4)]{convenient}, this is the pushout of $A \times S^n$ and $C \times_B A$ along $A$. Since all three of these spaces are $h$-fibrant over $B$ and the map $A \to (A \times S^n)$ has the homotopy extension property, this pushout is also $h$-fibrant over $B$ by \cite[Prop 2.1.9 and Lem 2.1.10]{convenient}.

In summary, $\barsmash_B$ preserves stable equivalences on all of our \fibrant objects.
\end{proof}

\begin{cor}
	$\Ex_B$ has a shadowed $n$-Fuller structure and a compatible system of base change objects indexed by $\cat{Fib}_B$.
\end{cor}

\subsection{Structures on the bicategory $G\Ex_B$}\label{sec:gexb}

Finally, we fix both a finite group $G$ and a base space $B$ with a $G$-action, and consider the bicategory $G\Ex_B$, defined the same was as $\Ex_B$ but with $G$-spaces and orthogonal $G$-spectra. So the base category is $G\cat{Fib}_B$, the category of equivariant Hurewicz fibrations $E \to B$. (An equivariant Hurewicz fibration satisfies the lifting condition for a Hurewicz fibration in the category of all $G$-spaces.) We have a symmetric monoidal bifibration $G\Osp_{(B)}$ of all orthogonal $G$-spectra $X$ over $G$-spaces $E$ with equivariant Hurewicz fibrations $E \to B$, whose monoidal product is the external smash product rel $B$, denoted $\barsmash_B$. From this we construct a point-set bicategory $G\Osp^c_{(B)}/G\cat{Fib}_B$ and a homotopy bicategory $G\Ex_B$ just as before.

In particular, a 0-cell in $G\Ex_B$ is a $G$-space $A$ with an equviariant Hurewicz fibration $A \to B$. A 1-cell from $(A \to B)$ to $(C \to B)$ is a freely $f$-cofibrant orthogonal $G$-spectrum over the fiber product $A \times_B C$, and a 2-cell is a map in the homotopy category of $G\Osp^c(A \times_B C)$, again using the ``genuine'' version of stable equivalence from \cite[7.4]{convenient}. The proof of the following is the same as in the previous subsection:

\begin{cor}
	$G\Ex_B$  has a shadowed $n$-Fuller structure and a compatible system of base change objects indexed by $G\cat{Fib}_B$.
\end{cor}

\section{Geometric fixed points and reduction of group action}\label{sec:functors}

Let $G$ be any finite group. For each subgroup $H \leq G$ and $G$-space $A$, let $A^H$ denote the subspace of $H$-fixed points. The $G$-action on $A$ descends to an action of the Weyl group $WH = NH/H$ on the fixed point space $A^H$.

In this section we show the geometric fixed points functor (defined below) and the forgetful functor induce strong shadow functors
\[ \Phi^H\colon G\Ex \to WH\Ex, \qquad \iota_H^*\colon G\Ex \to H\Ex \]
that commute up to vertical natural isomorphism (also known as an invertible icon, see \cref{df_iicon}) with the base change objects. We then do the same for the bicategories of spectra over spaces over $B$, where $B$ is a fixed space with trivial $G$-action:
\[ \Phi^H\colon G\Ex_B \to WH\Ex_B, \qquad \iota_H^*\colon G\Ex_B \to H\Ex_B. \]
This is the last result needed for \cite[Thm 8.3]{mp1}.

This section will mimic \cref{sec:rigid,sec:point_set,sec:deform,apply_parametrized_spectra}. We first prove a rigidity result (\cref{sec:rigid_geo_fp}), which requires us to get a little hands-on because it goes beyond the foundations for $G$-spectra covered in \cite{convenient}. Then we produce the point-set functors and natural transformations (\cref{sec:ssf_geo_fp,sec:nt_geo_fp}), and finally show this structure descends to the homotopy category (\cref{sec:htpy_geo_fp}).

\subsection{Rigidity}\label{sec:rigid_geo_fp}
We start by proving a rigidity result for geometric fixed points $\Phi^H$ that parallels \cref{rigidity}.

\begin{defn}\cite[V\S4]{mandell_may}, \cite[\S 9]{mp1}\label{df:geo_fp}
For a finite group $G$ and a subgroup $H \leq G$, the {\bf geometric fixed points functor} $\Phi^H\colon G\Osp \to WH\Osp$ sends each parametrized $G$-spectrum $X$ to the coequalizer
\[ \bigvee_{V,W} F_{W^{H}} S^0 \barsmash \mathscr J_G^{H}(V,W) \barsmash X(V)^{H} \rightrightarrows \bigvee_V F_{V^{H}} S^0 \barsmash X(V)^{H}  \ra \Phi^{H} X. \]
Here $V,W \subseteq \mc U$ run over finite-dimensional $G$-representations in some fixed complete $G$-universe $\mc U$, $\mathscr J_G$ is the indexing category for orthogonal $G$-spectra, $\mathscr J_G^H$ denotes its $H$-fixed points, and $F_{V^H}$ denotes the free spectrum at level $V^H$.
\end{defn}
Note that if $X$ is a $G$-spectrum over a space $A$, the geometric fixed points $\Phi^H X$ is a $WH$-spectrum over $A^H$.  See \cite[Prop 9.3]{mp1} for further discussion of geometric fixed points in the context of parametrized spectra.
 
\begin{lem} \label{geomfp_preserves}\cite[Prop 9.3]{mp1} 
The geometric fixed points functor 
commutes up to natural isomorphism with smash product, pullback, and pushforward on all freely $f$-cofibrant spectra.  
In particular, we have isomorphisms 
\begin{align}
\Phi^H(X\barsmash Y)&\cong \Phi^HX\barsmash\Phi^HY\label{phi_barsmash}\\
\Phi^H(f^*Y)&\cong (f^H)^*\Phi^H(Y)\label{phi_pullback}\\
\Phi^H(f_!Y)&\cong (f^H)_!\Phi^H(Y).\label{phi_pushforward}
\end{align}
\end{lem}
\begin{rmk}
	\cref{geomfp_preserves} requires that we use the (CG) convention from \cite[p. 6]{convenient}. In other words, the levels of our spectra $X_n$ are only assumed to be compactly generated, not compactly generated weak Hausdorff.\footnote{It appears to also hold in the (CGWH) convention (for freely $f$-cofibrant spectra), but we do not provide a proof here.}
\end{rmk}

For each multi-span of $G$-spaces and equivariant maps
\[ \xymatrix @R=5pt { 
		&B \ar[dl]_-f \ar[dr]^(.5){(g_1,\ldots,g_n)}
		\\
		C && A_1 \times \ldots \times A_n,
	}\]
we can take the action of the multi-span followed by the geometric fixed points functor:
\begin{equation}\label{phi_h_action}
\xymatrix @R=5pt {
	\prod_i G\Osp^c(A_i) \ar[r] & WH\Osp^c(C) \\
	(X_1,\ldots,X_n) \ar@{|->}[r] & \Phi^H f_!(g_1,\ldots,g_n)^*(X_1 \barsmash \ldots \barsmash X_n) \\
	& \cong f^H_!(g^H_1,\ldots,g^H_n)^*(\Phi^H X_1 \barsmash \ldots \barsmash \Phi^H X_n).
}
\end{equation}

\begin{thm}\label{phi_h_rigidity}
	When $H = G$, the functor $\Phi^G \circ f_! \circ (g_1,\ldots,g_n)^* \circ \barsmash$ from \eqref{phi_h_action} admits no nonidentity automorphisms. For general $H$, the identity is the only automorphism that is natural when the morphisms are extended to the maps that are only $H$-equivariant, i.e. the morphisms in $\prod_i H\Osp^c(A_i)$.
\end{thm}

\begin{proof}
	Without loss of generality $H = G$. 
Let $\eta$ be any automorphism of the functor 
\[f^G_!(g^G_1,\ldots,g^G_n)^*(\Phi^G X_1 \barsmash \ldots \barsmash \Phi^G X_n).\]
 By the proof of \cite[Thm 2.5.5]{convenient}, $\eta$ must be the identity on any $n$-tuple of suspension spectra $(\Sigma^\infty_{+A_1} (*),\ldots,\Sigma^\infty_{+A_n} (*))$, on retractive $G$-spaces of the form $(*)_{+A_i} = (*) \amalg A_i$, where the point $(*)$ has trivial $G$-action and therefore maps to $A_i^G \subseteq A_i$. 
	
	The proof of \cite[Prop 3.17]{malkiewich_thh_dx} carries over to the parametrized setting, as in \cite[Thm 4.5.2]{convenient}, and shows that $\eta$ must be the identity on any $n$-tuple of spectra 
\[(F_{V_1} (*)_{+A_1}, \ldots, F_{V_n} (*)_{+A_n})\]
 formed by taking free spectra on nontrivial representations on the same one-point retractive spaces $(*)_{+A_i}$.
	
	To further extend the class of spectra on which $\eta$ is the identity, note that if $f\colon X \to Y$ is a levelwise surjection of spectra, and $\psi$ is an automorphism commuting with this map
	\[ \xymatrix @R=1.7em{
		X \ar[r]^-\psi \ar[d]_-f & X \ar[d]^-f \\
		Y \ar[r]^-\psi & Y,
	} \]
if $\psi$ is the identity on $X$, then $\psi$ must also be the identity on $Y$. Now observe that $\Phi^G$ sends maps of spectra $X \to Y$ for which $X(V)^G \to Y(V)^G$ is surjective, to levelwise surjections of spectra. Furthermore, $\barsmash$, $g^*$, and $f_!$ all preserve levelwise surjections of spectra. So if we choose a map of tuples 
\[(X_1,\ldots,X_n) \to (Y_1,\ldots,Y_n)\]
 such that each map $X_i(V)^G \to Y_i(V)^G$ is a surjection, and if $\eta$ is the identity on $(X_1,\ldots,X_n)$, then $\eta$ must also be the identity on $(Y_1,\ldots,Y_n)$.
	
	If more generally $\eta$ is the identity on a collection of such tuples $(X_1^\alpha,\ldots,X_n^\alpha)$, and we choose maps $X_i^\alpha \to Y_i$ such that the maps $X_i^\alpha(V)^G \to Y_i(V)^G$ are jointly surjective, then taking their wedge sum gives a levelwise surjective map $\bigvee_\alpha X_i^\alpha \to Y_i$, 
which then $\Phi^G$, $\barsmash$, $g^*$, and $f_!$ take to a levelwise surjective map of spectra. So if $\eta$ is the identity on all of the tuples $(X_1^\alpha,\ldots,X_n^\alpha)$, then $\eta$ must also be the identity on $(Y_1,\ldots,Y_n)$.
	
	We now apply this principle to finish the proof. For any $n$-tuple of parametrized spectra $(X_1,\ldots,X_n)$, the free-forget adjunction gives a collection of maps of spectra 
\[F_V (*)_{+A_i} \to X_i,\]
 one for each representation $V$ and $G$-fixed point in $X_i(V)$. These maps are jointly surjective onto $X_i(V)^G$ by construction. Therefore, by the above discussion, since $\eta$ is the identity on every tuple of free spectra on one-point spaces, it must also be the identity on $(X_1,\ldots,X_n)$.	
\end{proof}

\subsection{Strong shadow functors}\label{sec:ssf_geo_fp}
Suppose $\mathcal{B}$ is a bicategory with shadow (\cref{df:shadow}) taking values in the 1-category $\cat{T}$. We call $\cat{T}$ the {\bf shadow category} for $\mathcal{B}$.

\begin{defn}\label{df:ssfun}
Following \cite{p:thesis}, suppose $\mathcal{B}$ and $\mathcal{B}'$ are bicategories with shadow categories $\cat{T}$ and $\cat{T}'$ respectively. A {\bf strong shadow functor} is a pseudofunctor $F\colon \mathcal{B}\to \mathcal{B}'$, a functor $F_{\mathrm{tr}}\colon \cat{T}\to \cat{T}$ and a natural isomorphism  $\sh{F(-)}\to F_{\mathrm{tr}}(\sh{-})$ so that 
the following diagram commutes.
\[\xymatrix{\sh{FX\odot FY}\ar[r]^\theta\ar[d]&\sh{FY\odot FX}\ar[r]&\sh{F(Y\odot X)}\ar[d]
\\
\sh{F(X\odot Y)}\ar[r]&F_{\mathrm{tr}}\sh{X\odot Y}\ar[r]^{F_{\mathrm{tr}}(\theta)}&F_{\mathrm{tr}}\sh{Y\odot X}}\]
\end{defn}

We now use the geometric fixed points functor $\Phi^H$ from \cref{df:geo_fp} and its rigidity theorem (\cref{phi_h_rigidity}) to define a strong shadow functor on the point-set bicategory of parametrized $G$-spectra.

\begin{prop}\label{prop:geo_ssf}
The geometric fixed points functor $\Phi^H\colon G\Osp \to WH\Osp$ extends to a strong shadow functor 
\[ \Phi^H\colon G\Osp^c/G\Top \to WH\Osp^c/WH\Top \] defined
on 0-cells, hom categories, and shadow categories as follows:  
\begin{itemize}
	\item On objects, the image of a $G$-space $A$ is the $WH$-space $A^H$.
	\item 
	On morphism categories, we take the functors
	\[ \Phi^H\colon G\Osp^c(A \times B) \to WH\Osp^c(A^H \times B^H). \]
	To be precise, $\Phi^H$ produces spectra over $(A \times B)^H$, and we pull back along the canonical identification to make these into spectra over $A^H \times B^H$.
	\item On shadow categories, we take the functor
	\[ \Phi^H\colon G\Osp^c(*) \to WH\Osp^c(*). \]
\end{itemize}
\end{prop}

\begin{proof}
The natural isomorphisms
	\begin{align*} m_{\Phi^H}\colon \Phi^H(X) \odot \Phi^H(Y) &\cong \Phi^H(X \odot Y) \\	
	 i_{\Phi^H}\colon U_{A^H} &\cong \Phi^H(U_A) \\
	 s_{\Phi^H}\colon \sh{\Phi^H(X)}& \cong \Phi^H\sh{X} 
	\end{align*}
	on freely $f$-cofibrant spectra arise by commuting $\Phi^H$ past the smash product, pullback, and pushforward using the isomorphisms in \eqref{phi_barsmash}, \eqref{phi_pullback}, and  \eqref{phi_pushforward}.
In each case this is the unique isomorphism that is natural for all $H$-equivariant maps, by  \cref{phi_h_rigidity}. 

To show these define a strong shadow functor, we must check the coherences in \cref{fig:strong_shadow}. The first three verify that $\Phi^H$ is a pseudofunctor, and the last one is from \cref{sec:ssf_geo_fp}.
\begin{figure}[h]
	\hspace{2em}
	\begin{subfigure}[t]{18em}
		\resizebox{\textwidth}{!}{
			\xymatrix{
				U_{A^H} \odot \Phi^H X \ar[d]_-{i_{\Phi^H} \odot \id} \ar[r]^-\ell & \Phi^H X \\
				(\Phi^H U_A) \odot \Phi^H X \ar[r]^-{m_{\Phi^H}} & \Phi^H (U_A \odot X) \ar[u]_-{\Phi^H(\ell)}
			}
		}
		\caption{The shadow functor left unitor axiom}\label{fig:shadow_unit}
	\end{subfigure}
	\hspace{5em}
	\begin{subfigure}[t]{18em}
		\resizebox{\textwidth}{!}{
			\xymatrix{
				\Phi^H X \odot U_{B^H} \ar[d]_-{\id \odot i_{\Phi^H}} \ar[r]^r & \Phi^H X \\
				\Phi^H X \odot (\Phi^H U_B) \ar[r]^-{m_{\Phi^H}} & \Phi^H (X \odot U_B) \ar[u]_-{\Phi^H(r)}
			}
		}
		\caption{The shadow functor right unitor axiom}\label{fig:shadow_unit_2}
	\end{subfigure}
	\begin{subfigure}[t]{24em}
		\resizebox{\textwidth}{!}{
			\xymatrix{
				(\Phi^H X \odot \Phi^H Y) \odot \Phi^H Z \ar[r]^-\alpha \ar[d]_-{m_{\Phi^H}\odot \id} &
				\Phi^H X \odot (\Phi^H Y \odot \Phi^H Z) \ar[d]^-{\id\odot m_{\Phi^H}} \\
				\Phi^H (X \odot Y) \odot \Phi^H Z \ar[d]_-{m_{\Phi^H}} &
				\Phi^H X \odot \Phi^H (Y \odot Z) \ar[d]^-{m_{\Phi^H}} \\
				\Phi^H ((X \odot Y) \odot Z) \ar[r]^-\alpha &
				\Phi^H (X \odot (Y \odot Z))
			}
		}
		\caption{The shadow functor associator axiom }\label{fig:shadow_odot}
	\end{subfigure}
	\hspace{2em}
	\begin{subfigure}[t]{17em}
		\resizebox{\textwidth}{!}{
			\xymatrix{
				\sh{ \Phi^H X \odot \Phi^H Y } \ar[r]^-\theta \ar[d]_-{\sh{m_{\Phi^H}}} &
				\sh{ \Phi^H Y \odot \Phi^H X } \ar[d]^-{\sh{m_{\Phi^H}}} \\
				\sh{ \Phi^H (X \odot Y) } \ar[d]_-{s_{\Phi^H}} &
				\sh{ \Phi^H (Y \odot X) } \ar[d]^-{s_{\Phi^H}} \\
				\Phi^H \sh{ X \odot Y } \ar[r]^-\theta &
				\Phi^H \sh{ Y \odot X }
			}
		}
		\caption{The shadow functor rotator axiom }\label{fig:shadow_shadow}
	\end{subfigure}

	\caption{Commutative diagrams for a strong shadow functor.}\label{fig:strong_shadow}
\end{figure}

\begin{itemize}
\item 
	In \cref{fig:shadow_odot}, each of the objects is isomorphic to the action of the multi-span
	\[ \xymatrix @R=5pt { 
		&A \times B \times C \times D\ar[dl]_-{1\pi \pi 1}\ar[dr]^-{1\Delta\Delta 1}
		\\
		A \times D && (A \times B) \times (B \times C) \times (C \times D)
	}\]
	(the one that defines the associator $\alpha$), acting with geometric fixed points as in \eqref{phi_h_action}. Here we are using that the geometric fixed points functor commutes up to isomorphism with smash product, pullback and pushforward on freely $f$-cofibrant spectra (\cref{geomfp_preserves}). So all of the functors in the diagram are of the form given in \eqref{phi_h_action}. Furthermore, all of the isomorphisms between them are chosen to be natural with respect to $H$-equivariant maps of the inputs. Therefore the diagram commutes by \cref{phi_h_rigidity}. 

	\item The coherences in \cref{fig:shadow_unit} and \cref{fig:shadow_unit_2} follow similarly using the 
	identity multi-span
	\[ \xymatrix @R=5pt { 
		&A \times B\ar[dl]_-{=}\ar[dr]^-{=}
		\\
		A \times B && A \times B.
	}\]
	\item The coherence in \cref{fig:shadow_shadow} 
		follows similarly using the multi-span that defined the rotator $\theta$,
	\[ \xymatrix @R=5pt { 
		&A \times B\ar[dl]_-{\pi \pi}\ar[dr]^-{\Delta\Delta}
		\\
		{\star} && (A \times B) \times (B \times A).
	}\]

\end{itemize}
\end{proof}

\subsection{Natural transformations}\label{sec:nt_geo_fp}  We now turn to compatibility between 
geometric fixed points and base change objects. We need to compare the pseudofunctors
\begin{equation}\label{geometric_fp_icon}
 \Phi^H \circ [] \textup{ and } [] \circ (-)^H\colon \ G\Top \to WH\Osp^c/WH\Top. 
\end{equation}
Here $G\Top$ is the 1-category of $G$-spaces and equivariant maps, regarded as a 2-category and therefore a bicategory in a trivial way, with only identity 2-cells.

The two pseudofunctors in \eqref{geometric_fp_icon} agree on 0-cells, so we can relate them by a vertical natural isomorphism (\cref{df_iicon}), equivalently an invertible icon in the sense of \cite{lack2010icons}. These are less complicated than pseudonatural transformations, and are exactly the compatibility that is required for \cite[Thm 8.3]{mp1}.

\begin{prop}\label{prop:geo_ssf_2}
There is a  vertical natural isomorphism (invertible icon) between the pseudofunctors $\Phi^H \circ []$ and $[] \circ (-)^H$ in \eqref{geometric_fp_icon}.
\end{prop}

\begin{proof}
The two pseudofunctors $\Phi^H \circ []$ and $ [] \circ (-)^H$ agree on objects, sending each $G$-space $A$ to the $WH$-space $A^H$.  There is an isomorphism of spectra over $A^H \times B^H$
\begin{equation}\label{icon:component} \eta\colon \Phi^H \left( \bcr{A}{f}{B} \right) \cong \bcr{A^H}{f^H}{B^H} ,
\end{equation}
which we see by commuting $\Phi^H$ past the smash product, pullback, and pushforward, using the isomorphisms in \eqref{phi_barsmash}, \eqref{phi_pullback}, and  \eqref{phi_pushforward}.
It is unique by the rigidity result in \cref{phi_h_rigidity}. 

To see that $\eta$ defines a vertical natural isomorphism, we write out \cref{df_iicon} in the case of \eqref{geometric_fp_icon}. Since the source bicategory is the 1-category $G\Top$ we can again dispense with \cref{icon_nat} and focus on \cref{icon_2_unit,icon_2_composition}, which become \cref{fig:icon_unit,fig:icon_odot}, respectively. Both of these diagrams commute because these base change objects admit no nontrivial automorphisms.
\begin{figure}[h]
	\begin{subfigure}[t]{27em}
		\resizebox{\textwidth}{!}{
		 \xymatrix{
		\Phi^H \left( \bcr{B}{g}{C} \right) \odot \Phi^H \left( \bcr{A}{f}{B} \right) \ar[d]^-{m_{\Phi^H}} \ar[r]^-{\eta \odot \eta} &
		\bcr{B^H}{g^H}{C^H} \odot \bcr{A^H}{f^H}{B^H} \ar[d]^-{m_{[]}} \\
		\Phi^H \left( \bcr{B}{g}{C} \odot \bcr{A}{f}{B}\right) \ar[d]^-{m_{[]}} &
		\bcr{A^H}{g^H \circ f^H}{C^H} \ar@{=}[d] \\
		\Phi^H \left( \bcr{A}{g \circ f}{C} \right) \ar[r]^-\eta &
		\bcr{A^H}{(g \circ f)^H}{C^H}
			}
		}
		\caption{The vertical composition axiom}\label{fig:icon_odot}
	\end{subfigure}
	\begin{subfigure}[t]{15em}
		\resizebox{\textwidth}{!}{
			\xymatrix{
				U_{A^H} \ar[d]^-{i_{\Phi^H}} \ar@{=}[r] & U_{A^H} \ar@{=}[d]^-{i_{[]}} \\
	\Phi^H(U_A) \ar@{=}[d]^-{i_{[]}} & \bcr{A^H}{=}{A^H} \ar@{=}[d] \\
	\Phi^H\left( \bcr{A}{\id}{A} \right) \ar[r]^-\eta & \bcr{A^H}{\id}{A^H}
			}
		}
		\caption{The vertical unit axiom}\label{fig:icon_unit}
	\end{subfigure}
	\caption{Commutative diagrams for an icon $\Phi^H \circ [] \to [] \circ (-)^H$.}\label{fig:icon}
\end{figure}
\end{proof}

\begin{rmk}
The reader may have noticed that we have not included anything about the compatibility between the geometric fixed points functor and the shadowed $n$-Fuller structure.  This is not because there is nothing to say, but because we have no need of it in \cite{mp1}. The techniques illustrated here can be used to prove any additional  compatibilities when needed.
\end{rmk}

\subsection{Reduction of groups}
We similarly define the strong shadow functor of point-set bicategories
\[ (\iota_H)^*\colon G\Osp^c/G\Top \to H\Osp^c/H\Top \]
and the vertical natural isomorphism (invertible icon) between the pseudofunctors
\[ (\iota_H)^* \circ [] \textup{ and } [] \circ (\iota_H)^*\colon \ G\Top \to H\Osp^c/H\Top \]
by forgetting the $G$-actions down to $H$-actions. Since this does not change the underlying set, we can pick the isomorphisms $m$, $s$, $i$, and $\eta$ to simply be identity maps. This makes the following a less complicated version of 
\cref{prop:geo_ssf,prop:geo_ssf_2}.

\begin{prop}\label{prop:forget_ssf}
	The above data make $(\iota_H)^*$ into a strong shadow functor, with a vertical natural isomorphism $(\iota_H)^* \circ [] \cong [] \circ (\iota_H)^*$.
\end{prop}

\begin{rmk}
	A rigidity theorem is not needed in this setting, but \cref{rigidity} applies here. It tells us that the functor that acts by any rigid multi-span, followed by the forgetful functor, has a unique automorphism that commutes with all non-equivariant maps.
\end{rmk}

\subsection{Descending to the homotopy category}\label{sec:htpy_geo_fp}
The process of showing that these functors descend to the homotopy category follows the same structure as in \cref{apply_parametrized_spectra}. We begin by recalling the following analog of \cref{convenient_main_thm}.

\begin{lem} \label{geom_fp_properties}\cite[Lem 9.6]{mp1} The geometric fixed points functor  preserves freely $f$-cofibrant spectra and stable equivalences between them.
\end{lem}
Unfortunately, we do not have a proof that $\Phi^H$ preserves level $h$-fibrancy for freely $f$-cofibrant spectra.\footnote{The issue is that the coequalizer destroys fibrancy. If we attempt to simplify the coequalizer to a colimit of pushouts of fibrant pieces, as in the proof of \cite[Thm 4.4.6]{convenient}, we would either have to use a free presentation, in which case we can't make the input spaces fibrant, or a semifree presentation, in which case we lose control of $\Phi^H$ because it does not have a nice formula on semifree inputs.} As a result, we cannot use the level $h$-fibrant spectra as our \fibrant objects, since they will not be preserved by $\Phi^H$. To get around this, we introduce a more restrictive class of ``$P$-fibrant'' spectra that are preserved by $\Phi^H$.

\begin{lem}\label{phi_p_commute}
For freely $f$-cofibrant spectra $X$, there is a natural isomorphism 
\[ \Phi^H PX \cong P\Phi^H X.\]
\end{lem}

\begin{proof}
	The functor $P$ is a composition of a pullback and a pushforward (\cite[\S 2.7]{convenient}), so this follows from the isomorphisms \eqref{phi_pullback} and \eqref{phi_pushforward} from \cref{geomfp_preserves}.
\end{proof}

\begin{defn}\label{p_fibrant}
	A spectrum $Y \in G\Osp^c(C)$ is {\bf $P$-fibrant} if it is isomorphic to a spectrum of the form $f_!g^* PX$, where $X \in G\Osp^c(A)$ (so $X$ is freely $f$-cofibrant), $g\colon B \to A$ is any map, and $f\colon B \to C$ is an equivariant Hurewicz fibration.
\end{defn}

\begin{lem}\label{p_fibrant_preserved}
$P$-fibrant spectra are preserved by
\begin{itemize}
	\item geometric fixed points $\Phi^H$,
	\item external smash product $\barsmash$,
	\item pullback $g^*$, and
	\item pushforward along any equivariant Hurewicz fibration $f_!$.
\end{itemize}
As a result, they are also preserved by $\odot$ and $\sh{}$.
\end{lem}

\begin{proof}
	Suppose $Y = f_!g^*X$ is $P$-fibrant. Then $\Phi^H Y = \Phi^H(f_!g^* PX) \cong f_!g^* P(\Phi^H X)$, using \cref{geomfp_preserves,phi_p_commute}. The spectrum $\Phi^H X$ is freely $f$-cofibrant, and therefore $\Phi^H Y$ is $P$-fibrant.
	
	Since the pushforward is the last functor applied in the definition of a $P$-fibrant spectrum,  a pushforward of a $P$-fibrant spectrum along an equivariant Hurewicz fibration is again $P$-fibrant. If we instead take a pullback $h^*f_!g^* PX$, we use the Beck-Chevalley isomorphism from \eqref{diagonal_square} to rewrite it as $j_!k^*g^* PX$, where $j$ is a pullback of $f$ and is therefore a fibration. Therefore a pullback of a $P$-fibrant spectrum is $P$-fibrant.
	
\cref{convenient_main_thm}, along with the observation that $\Osp$ is a symmetric monoidal bifibration, shows that for $P$-fibrant spectra $Y_1$ and $Y_2$ we have
	\begin{align*}
		Y_1 \barsmash Y_2 &\cong [(f_1)_!(g_1)^* PX_1] \barsmash [(f_2)_!(g_2)^* PX_2] \\
		&\cong (f_1 \times f_2)_!\left( [(g_1)^* PX_1] \barsmash [(g_2)^* PX_2]\right) \\
		&\cong (f_1 \times f_2)_!(g_1 \times g_2)^* \left( PX_1 \barsmash PX_2\right) \\
		&\cong (f_1 \times f_2)_!(g_1 \times g_2)^* P ( X_1 \barsmash X_2 ).
	\end{align*}
Since $X_1 \barsmash X_2$ is freely $f$-cofibrant,  $Y_1 \barsmash Y_2$ is $P$-fibrant.
\end{proof}

\begin{lem}\label{p_fibrant_equiv}
Every $P$-fibrant spectrum is level $h$-fibrant. Each of the operations in \cref{p_fibrant_preserved} preserves stable equivalences between $P$-fibrant spectra.
\end{lem}

\begin{proof}
	This follows directly from \cref{convenient_main_thm,geom_fp_properties}.
\end{proof}

We now show that the structures from \cref{prop:geo_ssf,prop:geo_ssf_2} descend to the homotopy category. Recall the notion of a strong shadow functor from \cref{df:ssfun}.
\begin{prop}\label{prop:geo_ssf_derived}
The strong shadow functor $\Phi^H$ from \cref{prop:geo_ssf} descends to a strong shadow functor on homotopy bicategories
\[ \Phi^H\colon G\Ex \to WH\Ex. \]
\end{prop}

\begin{proof}
On 0-cells we again take each $G$-space $A$ to the $WH$-space $A^H$. On morphism categories
	\[ \Ho G\Osp^c(A \times B) \to \Ho WH\Osp^c(A^H \times B^H), \]
	we take right-derived functor of $\Phi^H$. By \cref{geom_fp_properties}, $\Phi^H$ preserves all stable equivalences of freely $f$-cofibrant spectra, so the right-derived functor is $\Phi^H$ itself. We define the isomorphisms on the homotopy category
	\begin{align*} \ti m_{\Phi^H}\colon \Phi^H(X) \odot^\R \Phi^H(Y) &\simeq \Phi^H(X \odot^\R Y) \\	
	 \ti i_{\Phi^H}\colon U_{A^H} &\simeq \Phi^H(U_A) \\
	 \ti s_{\Phi^H}\colon \sh{\Phi^H(X)}^\R & \simeq \Phi^H\sh{X}^\R
	\end{align*}
from their point-set versions by checking that each composite of functors that appears is coherently right-deformable. For $\ti m_{\Phi^H}$ and $\ti s_{\Phi^H}$, we  take the tuples of $P$-fibrant spectra as our \fibrant objects. For $\ti i_{\Phi^H}$, we follow the  first part of \cref{ex_is_coherently_deformable_3} and include the unit objects $U_A \in \Osp^c(A \times A)$ and $U_{A^H} \in \Osp^c(A^H \times A^H)$ in the subcategory of \fibrant objects. \cref{rd_descends_2} then defines the maps $\ti m_{\Phi^H}$, $\ti i_{\Phi^H}$, and $\ti s_{\Phi^H}$.

We now check that each of the lists of functors in \cref{fig:strong_shadow} is coherently right-deformable. We illustrate the argument with 
the composite in the middle left entry in \cref{fig:shadow_odot}:
\begin{align*}
G\Osp^c(A \times B) \times G\Osp^c(B \times C) &\times G\Osp^c(C \times D)
 \xto{\odot \times \id} G\Osp^c(A \times C)\times G\Osp^c(C \times D) \\
& \xto{\Phi^H \times \Phi^H} WH\Osp^c(A^H \times C^H)\times WH\Osp^c(C^H \times D^H) \\
& \xto{\odot} WH\Osp^c(A^H \times D^H).
\end{align*}
In this case we take the \fibrant objects to be all tuples of $P$-fibrant spectra. They are preserved by \cref{p_fibrant_preserved}, and stable equivalences between them are preserved by \cref{p_fibrant_equiv}. 

We also illustrate the argument with the bottom left entry in \cref{fig:shadow_unit}:
\begin{align*}
(\ast)\times G\Osp^c(A \times B)
& \xto{U_A \times \id} G\Osp^c(A \times A)\times G\Osp^c(A \times B) \\
& \xto{\Phi^H \times \Phi^H} WH\Osp^c(A^H \times A^H)\times WH\Osp^c(A^H \times B^H) \\
& \xto{\odot} WH\Osp^c(A^H \times B^H).
\end{align*}
In this case we take the \fibrant objects to be all $P$-fibrant spectra, except in $G\Osp^c(A \times A)$ we also take $U_A$, and in $WH\Osp^c(A^H \times A^H)$ we also take $\Phi^H U_A$. Then the \fibrant objects are again preserved. To see that the stable equivalences between these \fibrant objects are preserved, for $\Phi^H$ we use the fact that it preserves all stable equivalences of freely $f$-cofibrant spectra (\cref{geom_fp_properties}), while for $\odot$ we use \cref{base_change_preserves_equivalences} and the fact that $\Phi^H U_A \cong U_{A^H}$ is isomorphic to a unit object.

We conclude by \cref{rd_descends_3} that the derived versions of the diagrams in \cref{fig:strong_shadow} also commute, defining a strong shadow functor on the homotopy category.
\end{proof}

\begin{rmk}
	As in \cref{just_put_in_a_ton_of_ps}, these isomorphisms in the homotopy category can be written explicitly by inserting copies of $P$ everywhere. This is even more true in this case since our fibrant objects are defined to be images of spectra of the form $PX$.
\end{rmk}

\begin{prop}\label{prop:geo_ssf_2_derived}
The vertical natural isomorphism from \cref{prop:geo_ssf_2} descends to a vertical natural isomorphism on homotopy categories between the pseudofunctors
\[ \Phi^H \circ [] \textup{ and } [] \circ (-)^H\colon \ \Ho G\Top \to WH\Ex. \]
\end{prop}

\begin{proof}
Recall the pseudofunctor $[]$ was defined in \cref{base_change_functor} and its coherence isomorphisms were derived in \cref{ex_is_coherently_deformable_4}. 

We take the isomorphism $\eta\colon \Phi^H \left( \bcr{A}{f}{B} \right) \cong \bcr{A^H}{f^H}{B^H}$ from \eqref{icon:component} and check that the derived versions of the diagrams in \cref{fig:icon} commute. Again, it suffices to check that each of the lists of functors that appears is coherently right-deformable. We illustrate the argument with the entry in the top left corner of \cref{fig:icon_odot}:
\begin{align*}
(\ast) \times (\ast)
& \xto{[g] \times [f]} G\Osp^c(C \times B)\times G\Osp^c(B \times A) \\
& \xto{\Phi^H \times \Phi^H} WH\Osp^c(C^H \times B^H)\times WH\Osp^c(B^H \times A^H) \\
& \xto{\odot} WH\Osp^c(C^H \times A^H).
\end{align*}
As in \cref{ex_is_coherently_deformable_4}, we take the \fibrant objects to be tuples in which 
\begin{itemize}
\item every spectrum in the tuple is $P$-fibrant or 
\item every spectrum in the tuple is the appropriate base change 1-cell. 
\end{itemize}
By construction, the \fibrant objects are preserved, and the stable equivalences between them are preserved by \cref{geom_fp_properties,base_change_preserves_equivalences}.

We conclude by \cref{rd_descends_3} that the derived versions of the diagrams in \cref{fig:icon} also commute, defining a vertical natural isomorphism on the homotopy category.
\end{proof}

\begin{prop}\label{prop:forget_ssf_derived}
The forgetful functor $(\iota_H)^*$ from \cref{prop:forget_ssf} descends to a strong shadow functor on the homotopy bicategory, with a vertical natural isomorphism $(\iota_H)^* \circ [] \cong [] \circ (\iota_H)^*$ as functors $\Ho G\Top \to H\Ex$.
\end{prop}

\begin{proof} Since $(\iota_H)^*$ preserves stable equivalences and fibrations \cite[Sections 7.2 and 7.4]{convenient},
the proof is the same as \cref{prop:geo_ssf_derived,prop:geo_ssf_2_derived}, but much less complicated. It is not even necessary to use $P$-fibrant spectra; the level $h$-fibrant spectra work perfectly well for this proof.
\end{proof}

\subsection{Geometric fixed points and change of group in the fiberwise case} Fix a base space $B$ with trivial $G$-action. Then the above results apply equally well to the bicategory $G\Ex_B$ of $G$-spectra over $G$-spaces $A$ with equivariant Hurewicz fibrations $A \to B$, which we first encountered in \cref{sec:gexb}. Each multi-span over $B$ as in \eqref{eq:multi_span_B} gives an action as in \eqref{action_of_multi_span_fiberwise}, but with the geometric fixed points $\Phi^H$ thrown in:
\begin{equation}\label{phi_h_action_fiberwise}
\xymatrix @R=5pt {
	\prod_i G\Osp^c(A_i) \ar[r] & WH\Osp^c(C) \\
	(X_1,\ldots,X_n) \ar@{|->}[r] & \Phi^H f_!(g_1,\ldots,g_n)^*(X_1 \barsmash_B \ldots \barsmash_B X_n) \\
	& \cong f^H_!(g^H_1,\ldots,g^H_n)^*(\Phi^H X_1 \barsmash_B \ldots \barsmash_B \Phi^H X_n).
}
\end{equation}
The rigidity theorem (\cref{phi_h_rigidity}) applies equally well to this new functor since it is also a composition of external smash products, pullbacks, pushforwards, and $\Phi^H$. As in \cref{sec:GExEx_B}, the proofs of \cref{prop:geo_ssf,prop:geo_ssf_2,prop:forget_ssf} are the same, except that the products become fiber products over $B$. We conclude that at the point-set level, $\Phi^H$ and $(\iota_H)^*$ are strong shadow functors that commute with base change objects up to a vertical natural isomorphism.
 
To pass to the homotopy category, we observe that the $P$-fibrant spectra of \cref{p_fibrant} are preserved by $\barsmash_B$ because it is a composition of $\barsmash$ and a pullback. Therefore the proofs of \cref{prop:geo_ssf_derived,prop:geo_ssf_2_derived,prop:forget_ssf_derived} proceed as before, using  \cref{fib_source_of_confusion} to verify that $\odot_B$ preserves stable equivalences between our choices of \fibrant objects in each case. (Note that the extra step in the proof of \cref{ex_b_is_coherently_deformable} does not appear here, because the structures we are interested in do not use the product $\boxtimes = \barsmash_B$ directly.) We conclude:
\begin{thm}\label{final}
	The geometric fixed points functor and the forgetful functor induce strong shadow functors
	\[ \Phi^H\colon G\Ex_B \to WH\Ex_B, \qquad \iota_H^*\colon G\Ex_B \to H\Ex_B, \]
	that commute with base change objects up to vertical natural isomorphism,
	\[\Phi^H \circ [] \cong [] \circ (-)^H \quad \textup{ and } \quad (\iota_H)^* \circ [] \cong [] \circ (\iota_H)^*. \]
\end{thm}

\bibliographystyle{amsalpha2}
\bibliography{references}%

\end{document}